\documentclass[11pt]{amsart}
\usepackage{mathrsfs} \textwidth=14.5cm \oddsidemargin=1cm
\usepackage{amsmath}
\usepackage{amsxtra}
\usepackage{amscd}
\usepackage{amsthm, color}
\usepackage{amsfonts}
\usepackage{amssymb}
\usepackage{eucal}
\usepackage{mathrsfs}

\input prepictex
\input pictex
\input postpictex

\newtheorem{thm}{Theorem}[section]
\newtheorem{cor}[thm]{Corollary}
\newtheorem{lem}[thm]{Lemma}
\newtheorem{prop}[thm]{Proposition}

\theoremstyle{definition}

\newtheorem{example}[thm]{Example}

\theoremstyle{remark}
\newtheorem{rem}[thm]{Remark}

\numberwithin{equation}{section}

\begin{document}

\newcommand{\thmref}[1]{Theorem~\ref{#1}}
\newcommand{\secref}[1]{Section~\ref{#1}}
\newcommand{\lemref}[1]{Lemma~\ref{#1}}
\newcommand{\propref}[1]{Proposition~\ref{#1}}
\newcommand{\corref}[1]{Corollary~\ref{#1}}
\newcommand{\remref}[1]{Remark~\ref{#1}}
\newcommand{\eqnref}[1]{(\ref{#1})}
\newcommand{\exref}[1]{Example~\ref{#1}}

\newcommand{\nc}{\newcommand}
\nc{\Z}{{\mathbb Z}} \nc{\C}{{\mathbb C}} \nc{\N}{{\mathbb N}}
\nc{\F}{{\mf F}} \nc{\Q}{\mathbb{Q}} \nc{\la}{\lambda}
\nc{\ep}{\epsilon} \nc{\h}{\mathfrak h} \nc{\n}{\mf n} \nc{\A}{{\mf
a}} \nc{\G}{{\mathfrak g}} \nc{\SG}{\overline{\mathfrak g}}
\nc{\DG}{\widetilde{\mathfrak g}} \nc{\D}{\mc D} \nc{\Li}{{\mc L}}
\nc{\La}{\Lambda} \nc{\is}{{\mathbf i}} \nc{\V}{\mf V}
\nc{\bi}{\bibitem} \nc{\NS}{\mf N}
\nc{\dt}{\mathord{\hbox{${\frac{d}{d t}}$}}} \nc{\E}{\mc E}
\nc{\ba}{\tilde{\pa}} \nc{\half}{\frac{1}{2}} \nc{\mc}{\mathcal}
\nc{\mf}{\mathfrak} \nc{\hf}{\frac{1}{2}}
\nc{\hgl}{\widehat{\mathfrak{gl}}} \nc{\gl}{{\mathfrak{gl}}}
\nc{\hz}{\hf+\Z} \nc{\vac}{|0 \rangle}
\nc{\dinfty}{{\infty\vert\infty}} \nc{\SLa}{\overline{\Lambda}}
\nc{\SF}{\overline{\mathfrak F}} \nc{\SP}{\overline{\mathcal P}}
\nc{\U}{\mathfrak u} \nc{\SU}{\overline{\mathfrak u}}
\nc{\ov}{\overline} \nc{\wt}{\widetilde} \nc{\sL}{\ov{\mf{l}}}
\nc{\sP}{\ov{\mf{p}}} \nc{\osp}{\mf{osp}} \nc{\sdeg}{\ov{\rm deg}}
\nc{\spo}{\mf{spo}} \nc{\hosp}{\widehat{\mf{osp}}}
\nc{\hspo}{\widehat{\mf{spo}}}

 \nc{\UU}{\mathscr{U}}
 \nc{\lbr}{[\![}
  \nc{\rbr}{]\!]}
  \nc{\VV}{\mathbb V}
 \nc{\WW}{\mathbb W}

\advance\headheight by 2pt

\title{ An inversion formula for some Fock spaces}

\author[Bintao Cao]{Bintao Cao$^\dagger$}
\thanks{$^\dagger$Partially supported by NSFC (Grant No. 11101436, 11571374)}
\address{School of Mathematics and Computational Science, Sun
Yat-sen University, Guangzhou, China 510275}
\email{caobt@mail.sysu.edu.cn}

\author[Ngau Lam]{Ngau Lam$^{\dagger\dagger}$}
\thanks{$^{\dagger\dagger}$Partially supported by an NSC-grant}
\address{Department of Mathematics, National Cheng-Kung University, Tainan, Taiwan 70101}
\email{nlam@mail.ncku.edu.tw}

\begin{abstract} \vspace{.3cm}
 A symmetric bilinear form on a certain subspace $\widehat{\mathbb T}^{\bf b}$ of a completion of the Fock  space $\mathbb T^{{\bf b}}$ is defined. The canonical and dual canonical bases of $\widehat{\mathbb T}^{\bf b}$ are dual with respect to the bilinear form.  As a consequence, the inversion formula connecting the coefficients of the canonical basis and that of the dual canonical basis of  $\widehat{\mathbb T}^{\bf b}$ expanded in terms of the standard monomial basis of $\mathbb T^{{\bf b}}$ is obtained.  Combining with the Brundan's algorithm for
computing  the elements in the canonical basis of
$\widehat{\mathbb{T}}^{{\bf b}_{\mathrm{st}}}$, we have an algorithm
computing the elements in the canonical basis of
$\widehat{\mathbb{T}}^{{\bf b}}$ for arbitrary ${\bf b}$.
\end{abstract}
\maketitle

\section{Introduction}  In his seminal paper \cite{Br} Brundan defined a symmetric bilinear form on the completion $\widehat{\mathbb T}^{{\bf b}_{\text{st}}}$  of the Fock  space $\mathbb T^{{\bf b}_{\text{st}}}:=\WW^{\otimes n}\otimes \VV^{\otimes m}$ showing that
the canonical and dual canonical bases of Lusztig and Kashiwara
\cite{Lu1, Ka} in $\widehat{\mathbb T}^{{\bf b}_{\text{st}}}$ are
dual with respect to the bilinear form, where $\VV$ is the natural
module of the quantum group $U_q({\mathfrak g\mathfrak l}_\infty)$
and $\WW$ is the restricted dual of $\VV$. As a consequence, the
inversion formula is obtained connecting the coefficients of the
canonical basis and that of the dual canonical basis of $\widehat{\mathbb T}^{{\bf b}_{\text{st}}}$ expanded in terms of the
standard monomial basis of $\mathbb T^{{\bf b}_{\text{st}}}$. For $n=0$,
we have $\widehat{\mathbb T}^{{\bf b}_{\text{st}}}= \mathbb T^{{\bf
b}_{\text{st}}}= \VV^{\otimes m}$ \cite{Br}, and the canonical and
dual canonical bases of $\VV^{\otimes m}$ can be identified with the
canonical and dual canonical bases of some modules of some Hecke
algebra of Type $A$ \cite{FKK} (see also \cite{Br3}). In \cite{Br3}, there is an elegant
proof showing that the inversion formula obtained in \cite{Br} for
$n=0$  equals the  inversion formula for relative
Kazhdan-Lusztig polynomials obtained in \cite{Do} (see also
\cite{So,KL}).

The Fock space $\mathbb T^{\bf b}$, which is a $q$-tensor space with
$m$ tensor factors isomorphic to $\VV$ and $n$ factors isomorphic to
$\WW$, determined by the $0^m1^n$-sequence $\bf b$ (see
\eqref{Bruhat} for a precise definition), was considered in
\cite{CLW2}. In analogy to Brundan's results, the canonical and dual
canonical bases  are defined in \cite{CLW2} in the subspace
$\widehat{\mathbb T}^{{\bf b}}$ of a completion of the Fock  space
$\mathbb T^{{\bf b}}$. In fact, the canonical basis is contained in
${\mathbb T}^{\bf b}$ (see Proposition~\ref{prop:can-finite} below).
In this paper, we generalize Brundan's definition defining  a
symmetric bilinear form on ${\widehat{\mathbb T}}^{{\bf b}}$  for
arbitrary $0^m1^n$-sequence ${\bf b}$ and show that the canonical
and dual canonical bases are dual with respect to the bilinear form.
The inversion formula  connecting the coefficients of the canonical
basis and that of the dual canonical basis of $\widehat{\mathbb T}^{{\bf b}}$ expanded in the standard
monomial basis of $\mathbb T^{{\bf b}}$ is obtained.  An equivalent version of the inversion formula has been obtained by Brundan, Losev and Webster \cite[Section 5.9]{BLW} via the garded tensor product categorifications. Let $t^{\bf b}_{gf}(q)$, called the Brundan-Kazhdan-Lusztig polynomial, denote the coefficient of the element $M_g^{\bf b}$  of the standard monomial basis in the expression for the element $T_f ^{\bf b}$ of the canonical basis. Using the
Brundan-Kazhdan-Lusztig conjecture proved in \cite{CLW2, BLW} and the positivity of the coefficients of the Brundan-Kazhdan-Lusztig polynomials proved in \cite{BLW, CLW2} (see also
Remark~\ref{positive} below), we show that there are only finitely
many $t^{{\bf b}}_{gf}(q)\not=0$ for a fixed $g$
or a fixed $f$. In particular, the canonical
basis $\{T^{\bf b}_f \}$ is contained in
$\mathbb{T}^{{\bf b}}$ (see Proposition~\ref{prop:can-finite} below). Also every element in the standard monomial basis of $\mathbb T^{{\bf b}}$ can be written as a finite sum of the elements in the dual canonical basis (see Proposition~\ref{prop:can-finite} below).

Let ${ \mathscr{E}}^{{\bf b}}$ (see \eqnref{def:q-wedge E} below) denote a $q$-wedge subspace of ${\mathbb T}^{{\bf b}}$ and let $\widehat{ \mathscr{E}}^{{\bf b}}$ (see Section~\ref{Sec:Can E} below for details) denote a certain subspace of ${\widehat{\mathbb T}}^{{\bf b}}$.
In Section 4, we define a symmetric bilinear form on $\widehat{ \mathscr{E}}^{{\bf b}}$ and showing that the canonical
basis and the dual canonical basis are dual with respect to the symmetric bilinear form. An inversion formula is also obtained.

In Section~5 we explain a method to compute the elements in the canonical
basis of $\widehat{\mathbb{T}}^{{\bf b}}$ from knowing the precise
expressions of the elements in the canonical basis of
$\widehat{\mathbb{T}}^{{\bf b}_\text{st}}$ in terms of standard
monomial basis. Combining with the Brundan's algorithm for
computing  the elements in the canonical basis of
$\widehat{\mathbb{T}}^{{\bf b}_{\mathrm{st}}}$ \cite[Section 2-j]{Br} (see also \cite[Section 3]{Br4}), we have an algorithm
computing the elements in the canonical basis of
$\widehat{\mathbb{T}}^{{\bf b}}$ for
arbitrary $0^m1^n$-sequence ${\bf b}$. We remark that the Brundan-Kazhdan-Lusztig polynomials arising as coefficients of the canonical bases also can  be computed as the approach implied by \cite{BLW} by truncating to finite parabolic Kazhdan-Lusztig polynomials and then by applying the classical algorithm for parabolic Kazhdan-Lusztig polynomials.

The paper is organized as follows. In Section 2 we review and develop some basic
results about the space $\widehat{\mathbb T}^{{\bf b}}$, the bar involution $\psi$ on $\widehat{\mathbb T}^{{\bf b}}$, and the canonical and dual canonical bases of $\widehat{\mathbb T}^{{\bf b}}$. In Section 3, we define a bilinear form on
${\widehat{\mathbb T}}^{{\bf b}}$  for arbitrary ${\bf b}$ and prove that
the bilinear form is symmetric. We show that the canonical and dual
canonical bases are dual with respect to the symmetric bilinear form. The
inversion formula is also described. In Section 4, we prove that the
similar results obtained in Section 3 are valid for the space
$\mathscr{E}^{\bf b}$. We explain the algorithm in Section 5.

\smallskip\smallskip
\noindent {\it Notations:} We shall use the following notations throughout this article. The symbols $\Z$,
$\N$, and $\Z_+$ stand for the sets of all integers, of positive integers and of non-negative integers, respectively.  For any subset $I\subseteq\Z$, denote the set of integer-valued functions on $I$ by $\Z^I$. For $r\in \N$, let $\lbr r\rbr:=\{1,2,\ldots,r\}$. $\Z^{\lbr r\rbr }$  is simply denoted by $\Z^{r}$ and each $f\in\Z^{r}$ is identified with the
$r$-tuple $\left(f(1),f(2),\ldots,f(r)\right)$ when convenient. Let $\Z^{r}_k$ denote the subset of $\Z^{r}$ consisting functions with values between $-k$ and $k$, where $k$ is  a positive integer.
For $f\in \Z^J$  and a subset $I\subseteq J$, we denote the restriction of $f$
to $I$ by $f_I$.

\section{Preliminaries}
\label{sec:Pre}

We will follow the notations and conventions as described in
\cite{CLW2}.  In this section, we review some definitions and
results about the canonical and dual canonical bases of the
completion $\widehat{{\mathbb T}}^{\bf b}$ of the Fock space
${\mathbb T}^{\bf b}$ obtained in \cite{CLW2} generalizing \cite{Br}.  We refer the reader to \cite[Part I]{CLW2}
for details. We also generalize some results in \cite{Br}  for the settings of \cite{CLW2}.

\subsection{Quantum group}\label{Qgp}

The quantum group $\UU:=U_q({\mathfrak g\mathfrak l}_\infty)$  with
an indeterminate $q$  is defined to be the associative algebra over the field
$\Q(q)$ of rational functions generated by $E_a, F_a, K_a, K^{-1}_a, a \in \Z$, subject to
the following relations ($a,b\in\Z$):
\begin{eqnarray*}
 K_a K_a^{-1} &=& K_a^{-1} K_a =1, \\
 K_a K_b &=& K_b K_a, \\
 K_a E_b K_a^{-1} &=& q^{\delta_{a,b} -\delta_{a,b+1}} E_b, \\
 K_a F_b K_a^{-1} &=& q^{\delta_{a,b+1}-\delta_{a,b}}
 F_b, \\
 E_a F_b -F_b E_a &=& \delta_{a,b} \frac{K_{a,a+1}
 -K_{a+1,a}}{q-q^{-1}}, \\
 E_a^2 E_b +E_b E_a^2 &=& (q+q^{-1}) E_a E_b E_a,  \quad \text{if } |a-b|=1, \\
 E_a E_b &=& E_b E_a,  \,\qquad\qquad\qquad \text{if } |a-b|>1, \\
 F_a^2 F_b +F_b F_a^2 &=& (q+q^{-1}) F_a F_b F_a,  \quad\, \text{if } |a-b|=1,\\
 F_a F_b &=& F_b F_a,  \qquad\ \qquad\qquad \text{if } |a-b|>1.
\end{eqnarray*}
Here $K_{a,a+1} :=K_aK_{a+1}^{-1}$. The co-multiplication $\Delta$
on $\UU$ is defined by:
\begin{eqnarray*}
 \Delta (E_a) &=& 1 \otimes E_a + E_a \otimes K_{a+1, a}, \\
 \Delta (F_a) &=& F_a \otimes 1 +  K_{a, a+1} \otimes F_a,\\
 \Delta (K_a) &=& K_a \otimes K_{ a}.
\end{eqnarray*}
The co-multiplication $\Delta$ here is consistent with the one used
by Kashiwara, but differs from \cite{Lu}. The counit $\epsilon$  is
defined by $\epsilon(E_a) =\epsilon(F_a) = 0$, $\epsilon(K_a) = 1$,
the antipode $S$ by $S(E_a) = -E_aK_{a,a+1}$, $S(F_a) =
-K_{a+1,a}F_a$, $S(K_a) = K_a^{-1}$. For $r\ge 1$, the divided
powers are defined by $E_a^{(r)} =E_a^{r}/[r]!$ and
$F_a^{(r)}=F_a^{r}/[r]!$, where $[r] =(q^r-q^{-r})/(q-q^{-1})$ and
$[r]!=[1][2]\cdots [r].$

In what follows we shall apply results from \cite{Lu} and
\cite{Jan}. To translate their results to our settings, we need to
replace $E_a$ with $F_a$ and $q$ therein by $F_a$, $E_a$ and
$q^{-1}$ for all $a\in\Z$, respectively, in order to match our
co-multiplication with theirs.

Setting $\ov{q}=q^{-1}$ induces an automorphism on $\Q(q)$ denoted
by $^{-}$. An {\em antilinear} map $f : V \longrightarrow W$ between
$\Q(q)$-vector spaces means that $f$ is a $\Q$-linear map such that
$f({cu})=\ov{c}f(u)$, for $c\in\Q(q)$ and $u\in V$. Define the bar
involution on $U_q({\mathfrak g\mathfrak l}_\infty)$ to be the
antilinear automorphism ${}^-: \UU\rightarrow \UU$ determined by
$\ov{E_a}= E_a$, $\ov{F_a}=F_a$, and $\ov{K_a}=K_a^{-1}$.

 For each
$k\in\mathbb{N}$, let $\UU_k:=U_q(\mathfrak{gl}_{|k|})$ denote the
subalgebra of $\UU$ generated by $\{E_a,
F_a,K^{\pm1}_a,K^{\pm1}_{a+1}\}$ for $-k\leq a\leq k-1$. Then
$\UU_k\subseteq \UU_{k+1}$ and $\bigcup_k \UU_k=\UU$. Let
$\UU^{\pm}_k$ and $\UU^{\pm}$ denote the positive and negative parts
of $\UU_k$ and $\UU$, respectively. Note that
$\UU^{\pm}_k\subseteq\UU^{\pm}_{k+1}$ and
$\bigcup_k\UU^{\pm}_k=\UU^{\pm}$.

Let $\texttt{P}$ denote the free abelian group on basis
$\{\varepsilon_a \,|\, a\in \Z\}$ endowed with a symmetric bilinear
form $(\cdot,\cdot)$ for which $\{\varepsilon_a|\, a\in \Z\}$ forms
an orthonormal basis. Let
$$\Pi=\{\alpha_a:=\varepsilon_a-\varepsilon_{a+1}\mid
a\in\mathbb{Z}\}\quad\text{and}\quad \Pi_k=\{\alpha_a\mid -k\leq
a\leq k-1\}$$
and let
$$
\Z_+\Pi:=\sum_{\alpha\in\Pi}\mathbb{Z}_+\alpha \quad\text{and}\qquad
\Z_+\Pi_k:=\sum_{\alpha\in\Pi_k}\mathbb{Z}_+\alpha.
$$
For $\mu\in \Z_+\Pi$, the corresponding $\mu$-weight space of
$\UU^+$ and $-\mu$-weight space of $\UU^-$ are defined by
$$\UU^\pm_{\mu}:=\{X\in \UU^\pm\mid K_aXK_a^{-1}=q^{(\pm\mu,\varepsilon_a)}X,\, \forall a\in \Z\},$$
respectively. Analogously we define $\mu$-weight space
$\UU^+_{k,\mu}$ of $\UU^+_k$ and $-\mu$-weight space $\UU^-_{k,\mu}$
of $\UU^-_k$ for $\mu\in \Z_+\Pi_k$. Note that
$$
\UU^+_{\mu}=\UU^+_{k,\mu}  \quad\text{and}\quad
\UU^-_{\mu}=\UU^-_{k,\mu} \quad \text{ for all } \mu\in \Z_+\Pi_k.
$$

\subsection{Quasi-$\mathcal{R}$-matrix}
In this subsection, we review an explicit description  of the
quasi-$\mathcal{R}$-matrix $\Theta$ \cite[Section 3.1]{CLW2}(cf.
\cite[Chapter~8]{Jan}). Proposition~\ref{sigma-theta} below will be used to show the bilinear form $\langle\cdot,\cdot\rangle$ on $\widehat{\mathbb{T}}^{\bf b}$ defined in Section~\ref{formula} is symmetric.

For $a\in\Z$, one can define an automorphism $T_a:\UU\rightarrow
\UU$ associated to $\alpha_a=\varepsilon_a-\varepsilon_{a+1}\in\Pi$
as follows  \cite[8.14]{Jan}:
\begin{equation*}\label{def:Ti1}
T_a(K_b)=\begin{cases}
K_{a},&\mbox{if}\ b=a+1,\\
K_{a+1},&\mbox{if}\ b=a,\\
K_b,&\mbox{otherwise};
  \end{cases}
\qquad  T_a(E_b)=\begin{cases}
    E_bE_a-q^{-1}E_aE_b, & \mbox{if}\ |b-a|=1,\\
    -K_{a+1,a}F_a,& \mbox{if}\ b=a,\\
    E_b, & \mbox{otherwise};
  \end{cases}
\end{equation*}
and
\begin{equation*}\label{def:Ti2}
 T_a(F_b)=\begin{cases}
    F_aF_b-qF_bF_a, & \mbox{if}\ |b-a|=1,\\
    -E_aK_{a,a+1},& \mbox{if}\ b=a,\\
    F_b, & \mbox{otherwise}.
     \end{cases}
\end{equation*}

For $k\in\N$, let $\mf{S}_{|k|}$ denote the symmetric group on the
set $\{-k,-k+1,\ldots,0,1,\ldots,k\}$, and let $w_0^{|k|}$ denote
the longest element in $\mf S_{|k|}$. Let
$\mf{S}_\infty:=\bigcup_{k}\mf{S}_{|k|}$ and let $s_a:=(a, a+1)$
denote the simple transposition in $\mf{S}_\infty$ for all $a\in
\Z$.  There exists an infinite sequence of integers
$\{a_i\}_{i=1}^\infty$ such that for each $k$ we have a reduced
expression for $w_0^{|k|}$ as follows \cite[(3.1)]{CLW2}:
\begin{equation}  \label{wn:wn+1}
w^{|k|}_0=s_{a_1}s_{a_2}\cdots s_{a_N},  \quad \text{where}\, N
=k(2k+1).
\end{equation}
 For
$t\in \N$, we define
\begin{equation*}  \label{theta:t}
\Theta_{[t]}:=\sum_{r\ge 0}q^{\hf r(r-1)}\frac{(q-q^{-1})^r}{[r]!}
T_{a_1}\cdots T_{a_{t-1}}(E^r_{a_t})\otimes T_{a_1}\cdots
T_{a_{t-1}}(F^r_{a_t}).
\end{equation*}
There is an explicit description for quasi-$\mc R$-matrix
$\Theta^{(k)}$ of $\UU_k$  as follow  \cite[8.30(2)]{Jan}:
\begin{align}  \label{eq:thetan}
\Theta^{(k)}=
\Theta_{[N]}\cdots\Theta_{[3]}\Theta_{[2]}\Theta_{[1]}\in
\sum_{\mu\in \Z_+\Pi_k}\UU_{k,\mu}^+\otimes\UU_{k,\mu}^-,  \quad
\text{ where } N =k(2k+1).
\end{align}
Actually, the quasi-$\mc R$-matrix $\Theta^{(k)}$ is an element in
some suitable completion of $\UU_k^+\otimes \UU_k^-$. We can write
\begin{align}\label{quasi:R:matrix, k}
\Theta^{(k)}=\sum_{\mu\in \Z_+\Pi_k}\Theta^{(k)}_{\mu},
\qquad\text{and }
\qquad\Theta^{(k)}_{\mu}=\sum_{i=1}^{\rm{dim}_{\Q(q)} \UU_{k,\mu}^+  }b_{\mu,i}'\otimes
b''_{\mu,i}\in \UU_{k,\mu}^+\otimes \UU_{k,\mu}^-,
\end{align}
where $b_{\mu,i}'\in\UU_{k,\mu}^+$ and
$b''_{\mu,i}\in\UU_{k,\mu}^-$. By the explicit description
\eqnref{eq:thetan} for the quasi-$\mc R$-matrix $\Theta^{(k)}$ , we
have \begin{equation*}\label{Theta-mu: stable}
\Theta^{(n)}_{\mu}=\Theta^{(k)}_{\mu}, \qquad\text{for all } n\ge
k\text{ and } \mu\in \Z_+\Pi_k.
\end{equation*}
Therefore we define $\Theta_{\mu}:=\Theta_{\mu}^{(k)}$ for any
$\mu\in \Z_+\Pi$ and $k\gg 0$. The \emph{quasi-$\mc R$-matrix} $\Theta$ for
$\UU$  is defined in \cite[Section 3.1]{CLW2} by
\begin{equation}\label{quasi:R:matrix}
\Theta=\sum_{\mu\in \Z_+\Pi}\Theta_{\mu}.
\end{equation}
Formally, $\Theta$ can be expressed by the infinite product
$\Theta=\cdots\Theta_{[3]}\Theta_{[2]}\Theta_{[1]}$.

Now we define the map $\sigma:\mathbb{Z}\longrightarrow\mathbb{Z}$
by $\sigma(a)=-1-a$ for all $a\in \Z$. Then $\sigma$ induces the map
$\sigma: \Pi\longrightarrow \Pi$ by
$\sigma(\alpha_a)=\alpha_{\sigma(a)}$ for all $\alpha_a\in \Pi$.
Following from \cite[Section 2-f]{Br}, $\sigma$ also induces the
antiautomorphism $\sigma: \UU\longrightarrow \UU$ defined by
\begin{equation*}
  \sigma(E_a)=E_{ \sigma(a)},\quad  \sigma(F_a)=F_{ \sigma(a)},\quad
  \sigma(K_a)=K_{-a},\qquad \text{ for all $a\in \Z$.}
\end{equation*}
Define the antiautomorphism $\widehat{\tau}: \UU\longrightarrow \UU$
by
\begin{equation*}
  \widehat{\tau}(E_a)=E_a,\quad \widehat{\tau}(F_a)=F_a,\quad
\widehat{\tau}(K_a)=K_a^{-1}, \qquad\text{ for all $a\in \Z$.}
\end{equation*}
Here the antiautomorphism $\widehat{\tau}$ is the antiautomorphism
$\tau$ defined in \cite[Lemma 4.6 b)]{Jan}. Note that $\sigma$ and $
\widehat{\tau}$ are involutions on $\UU$.

\begin{lem}\label{hat-tau-theta} For $\mu\in \Z_+\Pi$, we have
 \begin{equation}\label{eq:hat-tau-theta}
\widehat{\tau}\otimes\widehat{\tau}(\Theta_{\mu})=\Theta_{\mu}.
\end{equation}
\end{lem}

\begin{proof}
For $\mu\in \Z_+\Pi$, we choose $k$ large enough such that
$\Theta_{\mu}=\Theta_{\mu}^{(k)}$. Then \eqref{eq:hat-tau-theta} is
just the equality \cite[(7.1.2)]{Jan}.
\end{proof}

\begin{lem}\label{th1}  Let $\sigma':
\UU\longrightarrow \UU$ be the automorphism of  $\UU$ defined by
$\sigma'=\widehat{\tau}\circ\sigma=\sigma\circ\widehat{\tau}$. Then
we have
\begin{itemize}
\item[(i)]
$\sigma'\circ T_a(X)=T_{\sigma(a)}(\sigma'(X))$, for all $X\in \UU$
and $a\in \mathbb{Z}$,
\item[(ii)] $\sigma'(E_a)=E_{\sigma(a)},\ \ \sigma'(F_a)=F_{\sigma(a)}$,
for all $a\in\mathbb{Z}$.
\end{itemize}
\end{lem}

\begin{proof}
The statement (ii) is clear. Since $\sigma'$ and $T_a$ are
automorphisms, it is sufficient to show that the statement (i) holds
for $X=E_b, F_b, K_b$ with $b\in \Z$. Now the  lemma follows by some
straightforward computations.
\end{proof}

The following proposition will be helpful to show that the bilinear form $\langle\cdot,\cdot\rangle$ on $\widehat{\mathbb{T}}^{\bf b}$ defined in Section~\ref{formula} is symmetric.

\begin{prop}\label{sigma-theta} For $\mu\in \Z_+\Pi$, we have
 \begin{equation*}\label{5}
\sigma\otimes\sigma(\Theta_{\mu})=\Theta_{\sigma(\mu)}.
\end{equation*}
\end{prop}

\begin{proof}
By \eqnref{wn:wn+1}, the longest element $w_0^{|k|}$ in $\mathfrak{S}_{|k|}$  has a
reduced expression  $w_0^{|k|}=s_{a_1}s_{a_2}\cdots s_{a_N}$ and
hence the product $s_{\sigma(a_1)}s_{\sigma(a_2)}\cdots
s_{\sigma(a_N)}$ is also a reduced expression of $w_0^{|k|}$. Let
\begin{equation*}\label{4}
  \Theta_{[t],\sigma}:=\sum_{r\geq0}q^{\hf r(r-1)}\frac{(q-q^{-1})^r}{[r]!}T_{\sigma({a_1})}\cdots
  T_{\sigma(a_{t-1})}(E^r_{\sigma(a_t)})\otimes
   T_{\sigma({a_1})}\cdots T_{\sigma(a_{t-1})}(F^r_{\sigma(a_t)}).
\end{equation*}
By Lemma~\ref{th1}, we have
$\sigma'\otimes\sigma'(\Theta_{[t]})=\Theta_{[t],\sigma}$. Using \eqnref{eq:thetan} and  the fact that
product $s_{\sigma(a_1)}s_{\sigma(a_2)}\cdots s_{\sigma(a_N)}$ is a
reduced expression of $w_0^{|k|}$, we have $\sigma'\otimes\sigma'(\Theta^{(k)})=\Theta^{(k)}$. This implies
$\sigma'\otimes\sigma'(\Theta_{\mu})=\Theta_{\sigma(\mu)}$. By
Lemma~\ref{hat-tau-theta}, we have
$\sigma\otimes\sigma(\Theta_{\mu})=\sigma'\otimes\sigma'(\Theta_{\mu})=\Theta_{\sigma(\mu)}$.
\end{proof}

\subsection{Bruhat ordering}\label{Bruhat}
In this subsection, we review the Bruhat orderings $\preceq_{\bf b}$ on $\Z^{m+n}$ defined in \cite[Section 2.3]{CLW2} for any $0^m1^n$-sequence ${\bf b}$. For $r\in \N$,  let $\lbr r\rbr:=\{1,2,\ldots,r\}$ and let $\Z^{r}$ denote the set of
integer-valued functions on $\lbr r\rbr$. Each $f\in\Z^{r}$ is identified with the
$r$-tuple $\left(f(1),f(2),\ldots,f(r)\right)$ when convenient.

Recall that $\texttt{P}$ denote the free abelian group with
orthonormal basis $\{\varepsilon_r\vert r\in\Z\}$ with respect to a
bilinear form $(\cdot,\cdot)$ defined in Section~\ref{Qgp}. A partial order on $\texttt{P}$
defined by declaring $\nu\ge\mu$, for $\nu,\mu\in \texttt{P}$, if
$\nu-\mu$ is a non-negative integral linear combination of
$\varepsilon_r-\varepsilon_{r+1}$, $r \in \Z$.

For $m,n\in\Z_+$, a {\em $0^m1^n$-sequence} is a sequence ${\bf
b}=(b_1,b_2,\ldots,b_{m+n})$ of $m+n$ integers such that $m$ of the
$b_i$'s are equal to ${0}$ and $n$ of them are equal to ${1}$.
Fix a ${0^m1^n}$-sequence ${\bf b}=(b_1,b_2,\ldots,b_{m+n})$. For
$f\in\Z^{m+n}$ and $j\le m+n$, let
\begin{align*}
\text{wt}^j_{\bf b}(f)
 :=\sum_{j\le i}(-1)^{b_i}\varepsilon_{f(i)}\in \texttt{P},\quad
\text{wt}_{\bf b}(f):=\text{wt}^{1}_{\bf b}(f)\in \texttt{P}.
\end{align*}
Then the {\em Bruhat ordering of type ${\bf b}$} on $\Z^{m+n}$,
denoted by $\preceq_{\bf b}$, in terms of the partially ordered set
$(\texttt{P}, \leq)$ is defined in \cite[Section 2.3]{CLW2} as
follows: $g\preceq_{\bf b}f$ if and only if $\text{wt}_{\bf
b}(g)=\text{wt}_{\bf b}(f)$ and $\text{wt}^j_{\bf
b}(g)\le\text{wt}^j_{\bf b}(f)$, for all $j$. This is simply the
usual Bruhat ordering on the weight lattice $\Z^m$ of $\gl(m)$ if
$n=0$. The following lemma will be useful in the sequel.

\begin{lem}\cite[Lemma 2.4]{CLW2}\label{th:finite}
The poset $(\Z^{m+n},\preceq_{{\bf b}})$ satisfies the finite
interval property. That is, given $f,g$ with $g\preceq_{\bf b}f$,
the set $\{h\in\Z^{m+n}|g\preceq_{\bf b}h\preceq_{\bf b}f\}$ is
finite.
\end{lem}

\subsection{Fock space and its completion}\label{subsection:Fock space}
In this subsection, we review the Fock space ${\mathbb T}^{\bf b}$
and its B-completion ${\widehat{\mathbb T}}^{\bf b}$ associated to a
${0^m1^n}$-sequence ${\bf b}$ defined in \cite{CLW2} generalizing some results in \cite{Br}.
We deduce that the structure of $(\UU,\mc{H}^{{\bf b}})$-bimodule on $\mathbb{T}^{{\bf b}}$ \cite{Jim} extends to the completion $\widetilde{\mathbb{T}}^{{\bf b}}$ of $\mathbb{T}^{{\bf b}}$.

Let $\mathbb V$ be the natural $\UU$-module with basis
$\{v_a\}_{a\in\Z}$ and $\mathbb W :=\mathbb V^*$, the restricted
dual module of $\mathbb V$ with basis $\{w_a\}_{a\in\Z}$ such that $w_a(v_b):=
(-q)^{-a} \delta_{a,b}$. The actions of $\UU$ on $\mathbb V$ and
$\mathbb W$ are given by the following formulas:
\begin{align*}
&K_av_b=q^{\delta_{a,b}}v_b,\qquad E_av_b=\delta_{a+1,b}v_a,\ \
\quad
F_av_b=\delta_{a,b}v_{a+1},\\
&K_aw_b=q^{-\delta_{a,b}}w_b,\quad E_aw_b=\delta_{a,b}w_{a+1},\quad
F_aw_b=\delta_{a+1,b}w_{a}.
\end{align*}

Associate to a ${0^m1^n}$-sequence ${\bf b}$, the following tensor
space over $\Q(q)$ is called the {\em $\bf b$-Fock space} or simply
{\em Fock space}:
\begin{equation*}  \label{eq:Fock}
{\mathbb T}^{\bf b} :={\mathbb V}^{b_1}\otimes {\mathbb
V}^{b_2}\otimes\cdots \otimes{\mathbb V}^{b_{m+n}},\qquad
\text{where } {\mathbb V}^{b_i}:=\begin{cases}
{\mathbb V}, &\text{ if }b_i={0},\\
{\mathbb W}, &\text{ if }b_i={1}.
\end{cases}
\end{equation*}
The tensors here and in similar settings later on are understood to
be over the field $\Q(q)$. The algebra $\UU$ acts on $\mathbb T^{\bf
b}$ via the co-multiplication $\Delta$.

For $f=(f(1),\ldots,f(m+n))\in \Z^{m+n}$, the {\em standard monomial
basis} $\{M^{\bf b}_f |f\in \Z^{m+n}\}$ of ${\mathbb T}^{\bf b}$ is
defined by
\begin{equation}  \label{eq:Mf}
M^{\bf b}_f :=\texttt{v}^{b_1}_{f(1)}\otimes
\texttt{v}^{b_2}_{f(2)}\otimes\cdots\otimes
\texttt{v}^{b_{m+n}}_{f(m+n)},  \qquad\text{where } \texttt{v}^{b_i}:=\begin{cases}v,\text{ if }b_i={0},\\
w,\text{ if }b_i={1}.\end{cases}
\end{equation}
 We shall drop the superscript ${\bf b}$ for $M^{\bf b}_f$ if there is no confusion.

For a $0^m1^n$-sequence ${\bf b}=(b_1,\ldots,b_{m+n})$, there is  a
sequence of integer numbers $0=r_0<
r_1<r_2<\ldots<r_{d-1}<r_{d}=m+n$ with
\begin{align*}
I_1:=\lbr r_1\rbr,\qquad I_i:=\lbr r_i\rbr\setminus\lbr
r_{i-1}\rbr,\quad\text{for }  i=2,\cdots,d
\end{align*}
such that $b_{r_i}\not=b_{r_i+1}$ for $i\in\lbr d-1\rbr$ and
$b_i=b_j$  if $i,j\in I_k$, for all $k$. Associate to a
${0^m1^n}$-sequence ${\bf b}$ and $k\in\lbr d\rbr$, let
\begin{equation}\label{T:I}
{\mathbb T}^{{\bf b},{I_k}} :={\mathbb V}^{b_{r_{k-1}+1}}\otimes
{\mathbb V}^{b_{r_{k-1}+2}}\otimes\cdots \otimes{\mathbb
V}^{b_{r_{k}}}.
\end{equation}
Then we have
\begin{equation*}
{\mathbb T}^{{\bf b},{I_k}} =\begin{cases}\otimes_{i={r_{k-1}+1}}^{r_{k}} {\mathbb V},\quad\text{ if }b_k={0},\\
\otimes_{i={r_{k-1}+1}}^{r_{k}} {\mathbb W},\quad\text{ if
}b_k={1},\end{cases}
\end{equation*}
and
\begin{align*}
\mathbb T^{\bf b}= \mathbb T^{{\bf b},I_1}\otimes \mathbb T^{{\bf
b},I_2}\otimes\cdots\otimes \mathbb T^{{\bf b},I_d}.
\end{align*}
For  $k\in\lbr d\rbr$ and $f=(f(r_{k-1}+1),\ldots,f(r_k))\in
\Z^{I_k}$, the {\em standard monomial basis} $\{M^{{\bf b},{I_k}}_f
|f\in \Z^{I_k}\}$ of ${\mathbb T}^{{\bf b},{I_k}}$ is defined by
\begin{equation*}
M^{{\bf b},{I_k}}_f :=\texttt{v}^{b_k}_{f(r_{k-1}+1)}\otimes
\texttt{v}^{b_k}_{f(r_{k-1}+2)}\otimes\cdots\otimes
\texttt{v}^{b_{k}}_{f(r_{k})}.
\end{equation*}
For $f\in \Z^{m+n}$, we have
$$
M^{{\bf b}}_f= M^{{\bf b},{I_1}}_{f_{I_1}}\otimes M^{{\bf
b},{I_2}}_{f_{I_2}}\otimes\cdots\otimes M^{{\bf b},{I_d}}_{f_{I_d}}.
$$
 For $k\in\lbr d\rbr$, let $\mathfrak{S}_{I_k}$ denote the symmetric group  on the set
$I_k$ and let $\mc{H}_{I_k}$ denote the Iwahori-Hecke algebra
associated to $\mathfrak{S}_{I_k}$ generated by
$H_{r_{i-1}+1},\ldots,H_{r_{i}-1}$ subject to the relations
\begin{align*}
&(H_i -q^{-1})(H_i +q) = 0,\\
&H_i H_{i+1} H_i = H_{i+1} H_i H_{i+1},\\
&H_i H_j = H_j H_i, \quad\text{for } |i-j| >1.
\end{align*}
Let $s_{a}:=(a,a+1)$  denote the simple transposition in $\mathfrak{S}_{I_k}$ for $a\in I_k\backslash \{r_k\}$. For $x\in\mathfrak{S}_{I_k}$, we have the corresponding element
$H_{x}\in\mc{H}_{I_k}$, where $H_{x}=H_{i_1}\cdots H_{i_r}$ if
$x=s_{i_1}\cdots s_{i_r}$ is a reduced expression.
 The bar involution $\,\bar{ }\,$ on $\mc{H}_{I_k}$
is the unique antilinear automorphism defined by
$\overline{H_{x}}=H^{-1}_{x^{-1}}$ and $\overline{q}=q^{-1}$.

Let
\begin{equation*}\label{def:Hecke}
\mc{H}^{{\bf
b}}:=\mc{H}_{I_1}\otimes\mc{H}_{I_2}\otimes\cdots\otimes\mc{H}_{I_d},
\end{equation*}
and let  the bar involution $\,\bar{ }\,$ on $\mc{H}^{{\bf b}}$ be the antilinear involution
defined by
$$
\overline{Y_1\otimes Y_2\otimes\cdots\otimes
Y_d}:=\overline{Y_1}\otimes\overline{Y_2}\otimes\cdots\otimes\overline{Y_d},
\qquad \text {for all } Y_i,\in \mc{H}_{I_i}, \,  i=1,2, \cdots, d.
$$

 The algebra $\mc{H}^{{\bf b}}$ has a right action on $\mathbb
T^{\bf b}$ defined by
\begin{equation*}\label{def:H action}
  M_fH_i=\begin{cases}
   M_{f\cdot s_i}&\ \mbox{if}\ f\prec_{{\bf b}}f\cdot s_i,\\
   q^{-1}M_f &\ \mbox{if}\ f=f\cdot s_i,\\
   M_{f\cdot s_i}-(q-q^{-1})M_f &\ \mbox{if}\ f\succ_{{\bf b}}f\cdot
   s_i,
  \end{cases}
\end{equation*}
for $i\in\lbr m+n\rbr\setminus\{r_{1},\dots,r_d\}$, where  the group
$\mathfrak{S}^{{\bf
b}}:=\mathfrak{S}_{I_1}\times\cdots\times\mathfrak{S}_{I_d}$ acts on
the right on $\Z^{m+n}$ by composition of functions. This action
commutes with the left action of the quantum group $\UU$ and hence
$\mathbb T^{\bf b}$ becomes a $(\UU,\mc{H}^{{\bf b}})$-bimodule
\cite{Jim}.


Now we review the completions of Fock space $\mathbb{T}^{{\bf b}}$.
Let ${\bf b}$ be a fixed $0^m1^n$-sequence. For $k\in\N$,  let
$\mathbb{V}_{k}=\mbox{span}\{v_{-k},v_{-k+1},\ldots,v_{k}\}$ and
$\mathbb{W}_{k}=\mbox{span}\{w_{-k},w_{-k+1},\ldots,w_{k}\}$ be
$\mathbb{Q}(q)$-subspaces of $\mathbb{V}$ and $\mathbb{W}$,
respectively. It is easy to see $\mathbb{V}_{k}$ and
$\mathbb{W}_{k}$ are $\UU_k$-modules. The \emph{truncated Fock space} $ \mathbb{T}^{{\bf b}}_{\leq  |k|}$ is defined by
\begin{equation*}\label{eq:Tk}
  \mathbb{T}^{{\bf b}}_{\leq
  |k|}:=\mathbb{V}_{k}^{b_1}\otimes\cdots\otimes\mathbb{V}_{k}^{b_{m+n}},\qquad \text{where }  \mathbb{V}_{k}^{b_i}=\begin{cases}
\mathbb{V}_{k},&\mbox{if}\ b_i=0,\\
\mathbb{W}_{k},&\mbox{if}\ b_i=1.
\end{cases}
\end{equation*}
$\mathbb{T}^{{\bf b}}_{\leq  |k|}$ is also a $\UU_k$-module defined in the obvious way. For $k,r\in\mathbb{N}$, let $\mathbb Z^r_k$ denote the
subset of $\mathbb Z^r$ consisting of functions with values between $-k$ and $k$. It is clear that $\{M_f^{{\bf b}}|f\in \mathbb Z^{m+n}_k\}$ is a basis of  $\mathbb{T}^{{\bf b}}_{\leq
  |k|}$.

Let
\begin{equation*}\label{eq:pi_k}
\pi_{k}:\mathbb{T}^{{\bf b}}\longrightarrow\mathbb{T}^{{\bf
b}}_{\leq |k|}
\end{equation*} be the natural projection with respect to the
basis $\{M_f^{{\bf b}}\}$ for $\mathbb{T}^{{\bf b}}$. The kernels of
the $\pi_k$'s define a linear topology on the vector space
$\mathbb{T}^{{\bf b}}$. Let $\widetilde{\mathbb{T}}^{\bf b}$ denote
the completion of $\mathbb{T}^{\bf b}$ with respect to the linear
topology. Formally, every element in $\widetilde{\mathbb{T}}^{\bf
b}$ is a possibly infinite linear combination of $M_f$, for
$f\in\mathbb{Z}^{m+n}$. We let $\widehat{\mathbb{T}}^{{\bf b}}$
denote the subspace of $\widetilde{\mathbb{T}}^{\bf b}$ spanned by
elements of the form
\begin{equation*}\label{eq:completion}
  M_f+\sum_{g\prec_{{\bf b}} f}r_gM_g,\ \ \ \ \mbox{for}\ r_g\in\mathbb{Q}(q).
\end{equation*}
 The $\mathbb{Q}(q)$-vector spaces
  $\widetilde{\mathbb{T}}^{\bf b}$ and
  $\widehat{\mathbb{T}}^{{\bf b}}$ are called the \emph{A-completion}
  and \emph{B-completion} of $\mathbb{T}^{{\bf b}}$, respectively. It is easy to see that the actions of elements in $\UU$ and elements in $\mc{H}^{{\bf b}}$  on ${\mathbb{T}}^{{\bf b}}$ are continuous with respect  to the linear topology. Therefore the structure of $(\UU,\mc{H}^{{\bf b}})$-bimodule on $\mathbb{T}^{{\bf b}}$ \cite{Jim} extends to $\widetilde{\mathbb{T}}^{{\bf b}}$.

\subsection{Canonical and dual canonical bases}\label{subsection:canonical}
In this subsection, we review the definitions and results of the canonical and
dual canonical bases for $\widehat{\mathbb{T}}^{{\bf b}}$ obtained in \cite{CLW2} generalizing some results in \cite{Br}. We also generalized some results in \cite{Br}  for the settings of \cite{CLW2}.

The $\UU$-modules $\mathbb{V}$ and $\mathbb{W}$ have antilinear
\emph{bar} involutions, denoted by $\bar{}$, are defined by
$\overline{v}_a=v_a$ and $\overline{w}_a=w_a$, respectively, such that $\ov{u v_a}=\bar{u}\ov{v_a}$ and $\ov{u w_a}=\bar{u}\ov{w_a}$, for all $u\in U_q(\gl_\infty)$ and $a\in \Z$. Now we recall
the bar involution $\psi$ on $\widehat{\mathbb{T}}^{{\bf b}}$
defined in \cite[Section 3.3]{CLW2} as follows. First we define an antilinear involution $\psi^{(k)}$ on $\mathbb{T}^{{\bf b}}_{\leq|k|}$
inductively on $m+n$. For $m+n=1$, the antilinear involution $\psi^{(k)}$ is the restriction of bar involution on $\mathbb{T}^{{\bf b}}=\VV$ and $\mathbb{T}^{{\bf b}}=\WW$
for ${\bf b}=(0)$ and ${\bf b}=(1)$, respectively. Let ${\bf
b}=({\bf b}_1,{\bf b}_2)$ be  a $0^{m}1^{n}$-sequence such that ${\bf
b}_1$ and ${\bf b}_2$  are a $0^{m_1}1^{n_1}$-sequence
and a $0^{m_2}1^{n_2}$-sequence, respectively. By \cite[Section 27.3.1]{Lu}, there is an antilinear involution
$\psi^{(k)}$ on the truncated Fock space $\mathbb{T}^{{\bf b}}_{\leq|k|}=\mathbb{T}^{{\bf
b}_1}_{\leq|k|}\otimes \mathbb{T}^{{\bf b}_2}_{\leq|k|}$ defined via the
quasi-$\mathcal {R}$-matrix $\Theta^{(k)}$ by
\begin{equation*}\label{def:psi-k}
  \psi^{(k)}(u\otimes v):=\overline{u\otimes v}:=\Theta^{(k)}(\overline{u}\otimes\overline{v}), \ \
  \mbox{for}\ u\in \mathbb{T}^{{\bf b}_1}_{\leq|k|},\ v\in \mathbb{T}^{{\bf b}_2}_{\leq|k|}.
  \end{equation*}
Moreover, this definition is independent of the choice of ${\bf
b}_1$ and ${\bf b}_2$. Since $\psi^{(k)}({M_f})=\pi_k\left(\psi^{(l)}({M_f})\right)$ for all
$f\in\Z_k^{m+n}$ and $l\ge k$ \cite[Lemma 3.4]{CLW2}, we can define
the antilinear map
$ \psi:\,\mathbb{T}^{{\bf
b}}\longrightarrow\widetilde{\mathbb{T}}^{{\bf b}}$  \cite[(3.10)]{CLW2}, called the \emph{bar } map by
\begin{equation*}\label{def:psi}
 \psi(M_f):= \overline{M_f}:=\varprojlim_k\psi^{(k)}(M_f),\quad \text{for all}\,\, f\in\mathbb{Z}^{m+n}.
\end{equation*}
We have $\psi^{(k)}(M_f)=\pi_k(\psi(M_f))$. In
fact, $\psi(M_f)$  belongs to $\widehat{\mathbb{T}}^{{\bf b}}$ (see Proposition~\ref{prop:Mfbar} below).

  \begin{prop}\label{X-bar-H}
For all $f\in\Z^{m+n}$, $X\in \UU$ and $H\in\mc{H}^{{\bf b}}$, we
have
    $$\overline{XM_f H}=\overline{X}(\overline{M_f })\overline{H}.$$
\end{prop}

\begin{proof} $\overline{XM_f}=\overline{X}(\overline{M_f })$ follows from \cite[Lemma 7.1]{Jan} and induction on $m+n$. Note that $E_a$, $F_a$ and $q$ are replaced by $F_a$, $E_a$ and $q^{-1}$ in {\em loc.~cit.}, respectively. An elegant proof of $\overline{M_f H}=(\overline{M_f })\overline{H}$ can be found in \cite[Section 3 and 4]{Br3}.
\end{proof}

\begin{prop}\cite[Proposition 3.6, Lemma 3.7]{CLW2}\label{prop:Mfbar}
 For $f\in \mathbb{Z}^{m+n}$, we have
   \begin{equation*}
   \psi(M_f)=\overline{M_f}=M_f+\sum_{g\prec_{\bf b} f}r_{gf}(q)M_g,
  \end{equation*} where $r_{gf}$ in $\mathbb{Z}[q,q^{-1}]$ and the
  sum is possibly infinite.   Moreover, the bar map $ \psi$ on ${\mathbb T}^{\bf b}$ extends to $ \psi:\,\widehat{\mathbb T}^{\bf
b}\rightarrow\widehat{\mathbb T}^{\bf b}$ which is an involution.
\end{prop}

\begin{prop}\label{prop:can}\cite[Proposition 3.9]{CLW2}
The $\Q(q)$-vector space $\widehat{\mathbb T}^{\bf b}$ has unique
bar-invariant topological bases
$
\{T^{\bf b}_f|f\in\Z^{m+n}\}\text{ and }\{L^{\bf b}_f|f\in\Z^{m+n}\}
$
such that
\begin{align}\label{T-L}
T^{\bf b}_f=M^{\bf b}_f+\sum_{g\prec_{\bf b}f}t_{gf}^{\bf b}(q)
M^{\bf b}_g,
 \qquad
L^{\bf b}_f=M^{\bf b}_f+\sum_{g\prec_{\bf b}f}\ell_{gf}^{\bf b}(q)
M^{\bf b}_g,
\end{align}
with $t_{gf}^{\bf b}(q)\in q\Z[q]$, and $\ell_{gf}^{\bf b}(q)\in
q^{-1}\Z[q^{-1}]$, for $g\prec_{\bf b}f$. (We will also write
$t_{ff}^{\bf b}(q)=\ell_{ff}^{\bf b}(q)=1$, $t_{gf}^{\bf
b}=\ell_{gf}^{\bf b}=0$ for $g\npreceq_{\bf b}f$.)
\end{prop}

 We shall also drop the superscript ${\bf b}$ for $T^{\bf b}_f$, $L^{\bf b}_f$ , $t_{gf}^{\bf b}$ and $\ell_{gf}^{\bf b}$ if there is no confusion. $\{T_f|f\in\mathbb{Z}^{m+n}\}$ and $\{L_f|f\in\mathbb{Z}^{m+n}\}$
  are called the \emph{canonical basis} and \emph{dual canonical basis} for
  $\widehat{\mathbb{T}}^{{\bf b}}$, respectively. Also,
  $t_{gf}(q)$ and $\ell_{gf}(q)$ are called \emph{Brundan-Kazhdan-Lusztig
 polynomials}.

\begin{rem}\label{positive}The positive conjecture \cite[Conjecture
3.13]{CLW2}, which generalized \cite[Conjecture 2.28(i),(ii)]{Br}
for the standard $0^m1^n$- sequence ${\bf
b}_{st}=(1,\ldots,1,0,\ldots,0)$, said that $t^{\bf
b}_{gf}(q)\in\Z_+[q]$ and
 $\ell^{\bf b}_{gf}(-q^{-1})\in\Z_+[q]$. It has been proved by
Brundan, Losev and Webster \cite[Corollary 5.27]{BLW}. Another proof can be
found in  \cite[Remark 3.14]{CLW2}. The fact $t^{\bf
b}_{gf}(q)\in\Z_+[q]$ for all $f,g \in \Z^{m+n}$ is crucial  to show the finiteness of the
summation  for $T^{\bf b}_f$  in \eqref{T-L} (see Proposition~\ref{prop:t-poly-finite} and
 \ref{prop:can-finite} below).
\end{rem}

\subsection{Canonical and dual canonical bases of $\widehat{\mathscr{E}}^{\bf b}$}\label{Sec:Can E}
Recall that $\mathbb T^{\bf b}= \mathbb T^{{\bf b},I_1}\otimes
\mathbb T^{{\bf b},I_2}\otimes\cdots\otimes \mathbb T^{{\bf
b},I_d}$, $\mc{H}_{I_k}$  denotes the Iwahori-Hecke algebra
associated to the symmetric group  $\mathfrak{S}_{I_k}$ on the set
$I_k$ defined in Section~\ref{subsection:Fock space} for $k\in\lbr d\rbr$ and
$\mc{H}^{\bf{b}}=\mc{H}_{I_1}\otimes\mc{H}_{I_2}\otimes\cdots\otimes\mc{H}_{I_d}$.

Denoted by $w_{I_k}$ the longest element in $\mathfrak{S}_{I_k}$  for $k\in\lbr d\rbr$ , we
write
\begin{equation*}
  H_{I_k}:=\sum_{\omega\in\mathfrak{S}_{I_k}}(-q)^{\ell(w)-\ell(w_{I_k})}H_{w},
\end{equation*}
where $\ell(w)$ denotes the length of element $w$ in
$\mathfrak{S}_{I_k}$. For $x=(x_1, x_2,\cdots, x_d)\in  \mf{S}^{\bf
b}$,  we define $\ell(x):=\sum_{i=1}^d\ell(x_i)$. Let
\begin{equation}\label{def:w_0 H_0}
  w_0:=(w_{I_1},w_{I_2},\cdots, w_{I_d})\in \mf{S}^{\bf b},
\end{equation}and
\begin{equation*}
H_0:=H_{I_1}\otimes H_{I_2}\otimes\cdots\otimes H_{I_d}\in
\mc{H}^{\bf b}.
\end{equation*} We also let
$\tau:\mc{H}^{{\bf b}}\rightarrow\mc{H}^{{\bf b}}$ be the
antiautomorphism defined by $\tau(H_i)=H_i$, for any $ i\in\lbr
m+n\rbr\setminus\{r_{1},\dots,r_d\}$. Here and later on we identity $H_{x_i}$ with $1\otimes\cdots 1\otimes H_{x_i}\otimes1\otimes\cdots 1$ in the obvious way for  $H_{x_i}\in \mc{H}_{I_i}$. The following lemma summarizes
the elementary properties of $H_0$ which are easy to obtain (see
\cite[Lemma 3.2]{Br}).
\begin{lem}\label{lem:H0}
  The following properties hold:
\begin{itemize}
  \item[(i)]$H_iH_0=H_0H_i=-qH_0\ \mbox{for any }\ i\in\lbr
m+n\rbr\setminus\{r_{1},\dots,r_d\}$,
\item[(ii)]$\overline{H_0}=H_0$,
\item[(iii)]$H_0^2=-(\prod_{k=1}^d[r_k-r_{k-1}]!)H_0$,
\item[(iv)] $H_0=\tau(H_0)$.
\end{itemize}
\end{lem}

For each $k\in\lbr d\rbr$, we let $\mathscr{E}^{{\bf
b},I_{k}}:=\mathbb{T}^{{\bf b},I_k}H_{I_k}$ which is the
$q$-analogue of the exterior power $\bigwedge^{r_k-r_{k-1}} \mathbb
V^{b_{r_{k}}}$. Let
\begin{equation}\label{def:q-wedge E}
\mathscr{E}^{\bf b}:= \mathscr{E}^{{\bf b},I_1}\otimes
\mathscr{E}^{{\bf b},I_2}\otimes\cdots\otimes \mathscr{E}^{{\bf
b},I_{d}}.
\end{equation}
It is clear that $\mathscr{E}^{\bf b}= \mathbb{T}^{\bf b}H_0$. For
the case ${\bf b}_{\mathrm{st}}:=(1,\cdots,1,0,\cdots,0)$ consisting of $n$ ${1}$'s
followed by $m$ ${0}$'s, the definition of $\mathscr{E}^{{\bf b}_{\mathrm{st}}}$ goes back to
 \cite[Section 3.1]{Br}.

\begin{lem}\label{weakly dominant}
 Let $f\in\mathbb{Z}^{m+n}$
such that
\begin{equation}\label{W Dom}
 (-1)^{b_k} f(r_{k-1}+1)\geq  (-1)^{b_k} f(r_{k-1}+2)\geq\cdots\geq  (-1)^{b_k} f(r_k), \quad \hbox{for all $k\in \lbr d\rbr$}.
 \end{equation}
If  $ g\in\mathbb{Z}^{m+n}$ such that $g\preceq_{\bf b} f$,  then $g\cdot x\preceq_{\bf b} f$ for all $x\in \mf{S}^{\bf b}$.
\end{lem}

\begin{proof} It is sufficient to show for $x=s_i$, $ i\in\lbr
m+n\rbr\setminus\{r_{1},\dots,r_d\}$.  We need the following
characterization \cite[(2.3)]{CLW2} of $\preceq_{\bf b}$: for
$g,h\in\Z^{m+n}$,
\begin{equation*}
g\preceq_{\bf b}h\; \Leftrightarrow\; \sharp_{\bf b}(g,a,j)\le
\sharp_{\bf b}(h,a,j),\;
 \forall a\in\Z,j\in \lbr m+n \rbr, \text{
with equality for $j=1$,}
\end{equation*}
where
$$
\sharp_{\bf b}(h,a,j):=\sum_{j\le i,h(i)\le a}(-1)^{b_i}, \quad
\text{  for } h\in\Z^{m+n}.
$$
It is enough to show that $\sharp_{\bf b}(g\cdot s_i,a,i+1)\le
\sharp_{\bf b}(f,a,i+1)$ for all $a\in \Z$. We may assume that
$(-1)^{b_i}g(i)<(-1)^{b_i}g(i+1)$. Otherwise, we have $g\cdot
s_i\preceq_{\bf b} g\preceq_{\bf b} f$. Suppose that there is  $a
\in \Z$ such that $\sharp_{\bf b}(g\cdot s_i,a,i+1)> \sharp_{\bf
b}(f,a,i+1)$. This implies
\begin{equation*}
\sharp_{\bf b}(g\cdot s_i,a,i+2)=\sharp_{\bf b}(f,a,i+2) \quad
\text{and }\quad
\begin{cases}
g(i)\le a<f(i+1),&\mbox{for } b_i=0,\\
g(i)> a\ge f(i+1),&\mbox{for } b_i=1.
\end{cases}
\end{equation*}
Since $(-1)^{b_i}f(i)\ge (-1)^{b_i}f(i+1)$ and
$(-1)^{b_i}g(i)<(-1)^{b_i}g(i+1)$, we have
\begin{eqnarray*}
\sharp_{\bf b}(g,a,i)&\ge& \sharp_{\bf b}(f,a,i+2)+1=\sharp_{\bf b}(f,a,i)+1,\quad \text{ for }b _i=0;\\
 \sharp_{\bf b}(g,a,i)&\ge& \sharp_{\bf b}(f,a,i+2)-1 =\sharp_{\bf b}(f,a,i)+2-1, \quad\text{ for }b _i=1.
 \end{eqnarray*}
 Therefore $g\npreceq_{\bf b} f$ which contradicts to our assumption. Hence $g\cdot s_i\preceq_{\bf b} f$.
\end{proof}

\begin{prop}\label{H-module T}
$\widehat{\mathbb T}^{\bf b}$ is an $\mc{H}^{\bf b}$-module.
\end{prop}

\begin{proof} It is enough to show $uH_i\in \widehat{\mathbb T}^{\bf b}$
for $u=\sum\limits_{g\preceq_{\bf b} h}r_gM_g$ and $i\in\lbr
m+n\rbr\setminus\{r_{1},\dots,r_d\}$.  There is an $w\in \mf{S}^{\bf
b}$ such that $f:=h\cdot w$ satisfies the condition \eqref{W Dom}.
By Lemma~\ref{weakly dominant}, $h=f\cdot w^{-1}\preceq_{\bf b} f$.
Therefore we may assume $u=\sum\limits_{g\preceq_{\bf b} f}r_gM_g$.
Now $uH_i\in \widehat{\mathbb T}^{\bf b}$ follows from
Lemma~\ref{weakly dominant} and the fact that $\widetilde{\mathbb
T}^{\bf b}$ is an $\mc{H}^{\bf b}$-module.
\end{proof}

Let $f\in\mathbb Z^{m+n}$. It is called ${\bf b}$\emph{-dominant}, if the following inequalities hold
for all $k\in \lbr d\rbr$:
\begin{equation*}
 (-1)^{b_k}f(r_{k-1}+1)>\cdots>(-1)^{b_k}f(r_k).
\end{equation*}
It is called ${\bf b}$\emph{-antidominant}, if the following inequalities hold for all $k\in \lbr
d\rbr$:
\begin{equation*}
 (-1)^{b_k}f(r_{k-1}+1)\leq\cdots\leq (-1)^{b_k}f(r_k).
\end{equation*}
The set of all $\bf b$-dominant $f\in\mathbb Z^{m+n}$ is denoted by $\mathbb
Z^{{\bf b},+}$. For $f\in\mathbb Z^{{\bf b},+}$, we define
\begin{equation*}\label{def:Kf}
  K_f:=M_{f\cdot w_0}H_0\in\mathscr{E}^{{\bf b}}.
\end{equation*}
Recall that $w_0$ is defined in \eqref{def:w_0 H_0}. We will also
write $K_f:=0$ if $f\in\mathbb Z^{m+n}\backslash\mathbb Z^{{\bf
b},+}$. The following lemma is a generalization of \cite[Lemma
3.4]{Br}, which implies that $\{K_f\}_{f\in\mathbb Z^{{\bf b},+}}$
forms a basis of $\mathscr{E}^{\bf b}$.

\begin{lem}\label{lem:MfH0}
  Let $f\in\mathbb{Z}^{m+n}$ and $x$ be the unique element of
  minimal length in $\mathfrak{S}^{\bf b}$ such that $f\cdot x$ is $\bf b$-antidominant. Then
\begin{equation*}
  M_fH_0=\begin{cases}
    (-q)^{\ell(x)}K_{f\cdot x w_0}&\ \mbox{if}\ f\cdot x
    w_0\in\mathbb Z^{{\bf b},+},\\
    0 &\ \mbox{otherwise}.
  \end{cases}
\end{equation*}
\end{lem}

Let $\widehat{\mathscr{E}}^{\bf b}=\widehat{\mathbb T}^{\bf b}H_0$.
By Proposition~\ref{H-module T}, Lemma~\ref{weakly dominant} and \ref{lem:MfH0},
$\widehat{\mathscr{E}}^{\bf b}$ is a subspace of
$\widehat{\mathbb{T}}^{{\bf b}}$ spanned by elements of the form
\begin{equation*}\label{eq:K}
K_f+\sum_{g\prec_{{\bf b}} f,\ g\in \mathbb Z^{{\bf b},+}}r_gK_g,\quad\
f\in \mathbb Z^{{\bf b},+}\ \mbox{and}\ r_g\in\mathbb{Q}(q).
\end{equation*}

The bar involution $\psi$ on $\widehat{\mathbb T}^{\bf b}$ leaves
$\widehat{\mathscr{E}}^{\bf b}$ invariant by Lemma~\ref{lem:H0}(ii)
and Proposition~\ref{X-bar-H}.  By Lemma~\ref{weakly dominant} and
\ref{lem:MfH0}, we have
$$
\overline{K_f}=K_f+\sum_{g\prec_{{\bf b}}f,\ g\in \mathbb Z^{{\bf b},+}}s_{gf}(q)K_g, \quad \hbox{for all
${f\in\mathbb Z^{{\bf b},+}}$,}
$$
where $s_{gf}(q)\in\mathbb{Z}[q,q^{-1}]$. By \cite[Lemma 3.8]{CLW2} and
Lemma~\ref{th:finite}, we have the following proposition.
\begin{prop}
The $\mathbb{Q}(q)$-vector space $\widehat{\mathscr{E}}^{\bf b}$ has
unique bar-invariant topological bases
$\{\mc{U}_f|f\in\mathbb{Z}^{{\bf b},+}\}$ and
$\{L_f|f\in\mathbb{Z}^{{\bf b},+}\}$ such that
\begin{equation*}
  \mc{U}_f=K_f+\sum_{g\prec_{\bf b} f,\ g\in\mathbb Z^{{\bf b},+}}u_{gf}(q)K_g,
  \ \ \ \ L_f=K_f+\sum_{g\prec_{\bf b} f,\ g\in\mathbb Z^{{\bf b},+}}\ell_{gf}(q)K_g,
\end{equation*}
with $u_{gf}(q)\in q\mathbb{Z}[q]$ and $\ell_{gf}(q)\in
q^{-1}\mathbb{Z}[q^{-1}]$ for $g\prec_{\bf b} f$. (We will also
write $u_{ff}(q)=\ell_{ff}(q)=1$, $u_{gf}(q)=\ell_{gf}(q)=0$ for
$g\npreceq_{\bf b} f$ or $g\notin\mathbb Z^{{\bf b},+}$).
\end{prop}
  $\{\mc{U}_f|f\in\mathbb{Z}^{{\bf b},+}\}$ and $\{L_f|f\in\mathbb{Z}^{{\bf b},+}\}$
  are called the \emph{canonical basis} and \emph{dual canonical basis} for
  $\widehat{\mathscr{E}}^{\bf b}$, respectively.
  Using Brundan's arguments in the paragraph before \cite[Lemma 3.8]{Br},
  the elements $L_f$ and the polynomials $\ell_{gf}(q)$
  defined here are exactly the same as in Proposition~\ref{prop:can},
  for ${f,g\in\mathbb Z^{{\bf b},+}}$. The same proof of  \cite[Lemma 3.8]{Br} applies here to have the following proposition.

\begin{prop}
  For $f\in\mathbb{Z}^{{\bf b},+}$, $\mc{U}_f=T_{f\cdot w_0}H_0$.
\end{prop}

\section{Inversion formula}\label{formula}
In this section, we completely follow the strategy as in
\cite[Section 2-i]{Br} defining a symmetric bilinear form on $\widehat{\mathbb{T}}^{{\bf b}}$ and showing that the canonical
basis $\{T^{\bf b}_f|f\in\mathbb{Z}^{m+n}\}$ and the dual canonical basis
$\{L^{\bf b}_f|f\in\mathbb{Z}^{m+n}\}$ are dual with respect to the symmetric bilinear form.  As a consequence, we give an inversion formula connecting the coefficients of the canonical basis
and that of the dual canonical basis of $\widehat{\mathbb{T}}^{{\bf b}}$  expanded in terms of the standard monomial basis of $\mathbb T^{{\bf b}}$. The proof of the bilinear form being symmetric here is different from \cite{Br}. Using the Brundan-Kazhdan-Lusztig conjecture proved in \cite{CLW2, BLW} and the Brundan-Kazhdan-Lusztig polynomials $t^{\bf b}_{gf}(q)\in\Z_+[q]$ (see Remark~\ref{positive}), we show that there are only finitely many $t^{{\bf b}}_{gf}(q)\not=0$ for a fixed $g\in\mathbb{Z}^{m+n}$ or a fixed $f\in\mathbb{Z}^{m+n}$. In particular, the canonical basis $\{T^{\bf b}_f \,|\, f\in\Z^{m+n}\}$ is contained in  $\mathbb{T}^{{\bf b}}$. Also every element in the standard monomial basis of $\mathbb T^{{\bf b}}$ can be written as a finite sum of the elements in the dual canonical basis.

Let $(\cdot,\cdot)$ be the symmetric bilinear form on
$\mathbb{T}^{{\bf b}}$ defined by
\begin{equation*}\label{def:bilinear T}
(M_f,M_g)=\delta_{f,g},\qquad \text{for all } f,
g\in\mathbb{Z}^{m+n}.
\end{equation*}
The bilinear form $(\cdot,\cdot)$ on ${\mathbb{T}}^{{\bf b}}\times
{\mathbb{T}}^{{\bf b}}$ can be extended to a bilinear map
$(\cdot,\cdot)$ on $\widetilde{\mathbb{T}}^{{\bf b}}\times
{\mathbb{T}}^{{\bf b}}$ in an obvious way. Recall that
$\widetilde{\mathbb{T}}^{\bf b}$ is a completion of
$\mathbb{T}^{\bf b}$  defined in
Section~\ref{subsection:Fock space}.

We also define the antilinear map $\sigma:\mathbb{T}^{{\bf
b}}\longrightarrow\mathbb{T}^{{\bf b}}$ by
\begin{equation}\label{def:sigma on T}
\ \ \ \sigma(M_f)=M_{-f}, \qquad \text{for all }
f\in\mathbb{Z}^{m+n}.
\end{equation}
The antilinear map $\sigma$ can be extended to the antilinear map
$\sigma : \widetilde{\mathbb{T}}^{{\bf
b}}\longrightarrow\widetilde{\mathbb{T}}^{{\bf b}}$ but it may not
send elements in $\widehat{\mathbb{T}}^{{\bf b}}$ to
$\widehat{\mathbb{T}}^{{\bf b}}$.

Following from  \cite[Section 2-f]{Br}, let $\tau:
\UU\longrightarrow \UU$ be the antiautomorphism defined by
\begin{eqnarray}\label{1}\nonumber
\tau(E_a)=q^{-1}K_{a+1,a}F_a,\quad \tau(F_a)=qE_aK_{a,a+1},\quad
  \tau(K_a)=K_a,\quad \text{for all } f\in\mathbb{Z}^{m+n}.
\end{eqnarray}
Note that $\tau$ is an involution on $\UU$.

The following lemma is a generalization of \cite[Lemma 2.9, Theorem
2.14(ii)]{Br}. Recall that $\widetilde{\mathbb{T}}^{{\bf b}}$ is a
$(\UU, \mc{H}^{{\bf b}})$-bimodule.

\begin{lem}\label{th3}
For all $u\in\widetilde{\mathbb{T}}^{{\bf b}}$,
$v\in\mathbb{T}^{{\bf b}}$, $X\in \UU$ and $H\in\mc{H}^{{\bf b}}$,
we have
  \begin{itemize}
    \item[(i)] $(XuH,v)=(u,\tau(X)v\tau(H)),$
    \item[(ii)]
    $\sigma(XuH)=\tau(\overline{\sigma(X)})\sigma(u)\overline{H}$,
      \end{itemize}
\end{lem}

\begin{proof}  Let us first show that the lemma holds for $u\in\mathbb{T}^{{\bf b}}$. These are all checked directly for $\mc{H}^{{\bf b}}$. It is easily to check these statements hold for $\UU$ in the simple case $m+n=1$.
 Since $\tau$ and $-\circ\sigma$ are coalgebra automorphisms, the
 lemma follows by induction on $m+n$.

 Now we assume $u\in\widetilde{\mathbb{T}}^{{\bf b}}$. Part (i) follows from the first part of the proof and continuity of the actions of $X$ and $H$. Part (ii) follows from the first part of the proof and the uniqueness of the extension of a given  continuous map  on ${\mathbb{T}}^{{\bf b}}$ to  ${\widehat{\mathbb{T}}}^{{\bf b}}$ .
\end{proof}

Analogous to \cite[(2.20)]{Br},  we define the bilinear form
$\langle\cdot,\cdot\rangle$ on $\widehat{\mathbb{T}}^{{\bf b}}$ by
\begin{equation}\label{8}
 \langle u,v\rangle:=\sum_{g\in \Z^{m+n}}(u,M_g)( \sigma(\overline{v}),M_{g}),\qquad \forall\, u,v\in \widehat{\mathbb{T}}^{{\bf b}}.
\end{equation}
It is clear that the bilinear form
$\langle\cdot,\cdot\rangle$ can be interpreted as
 $$
\langle u,v\rangle=\sum_{g\in
\Z^{m+n}}(u,M_g)\overline{(\overline{v},M_{-g})},\qquad \forall\,
u,v\in \widehat{\mathbb{T}}^{{\bf b}}.
 $$

Let $u=\sum\limits_{g\preceq_{\bf b} f}r_gM_g,\, v=\sum\limits_{l\preceq_{\bf b} h}s_l M_l\in \widehat{\mathbb{T}}^{{\bf b}}$. We have
$\overline{v}=\sum\limits_{l\preceq_{\bf b} h}s'_l M_l$. By
Lemma~\ref{th:finite}, there only finitely many $g$ such that $M_g$
is involved in $u$ and $M_{-g}$ is involved in $\overline{v}$. This
implies all but finitely many terms on the right hand side of the
equation  \eqnref{8} are zero. Therefore  the bilinear form
$\langle\cdot,\cdot\rangle$ is well defined on $\widehat{\mathbb{T}}^{{\bf b}}$.

From the arguments above, we also have the following.

\begin{lem}\label{th2}
 For $u=\sum\limits_{g\preceq_{\bf b} f}r_gM_g,\, v=\sum\limits_{l\preceq_{\bf b} h}s_l M_l\in \widehat{\mathbb{T}}^{{\bf b}}$, we have
  \begin{equation}\label{9}
    \langle u,v\rangle
=\sum_{g,l}r_g s_l\langle M_g,M_l\rangle.
\end{equation}
Moreover, the right side of \eqref{9} is a finite sum.
\end{lem}

\begin{prop}\label{th:symmetric}
  The bilinear form $\langle\cdot,\cdot\rangle$ defined in \eqref{8}
  is symmetric.
\end{prop}

\begin{proof}
By Lemma \ref{th2}, it is sufficient to show
  \begin{equation}\label{12}
    \langle M_f,M_g\rangle=\langle M_g,M_f\rangle  \qquad \text{ for all}\
    f,g\in\mathbb{Z}^{m+n}.
\end{equation}
 We will prove \eqref{12} by induction on $m+n$. For $m+n=1$, it
is clear that  \eqref{12} holds. Now we assume that  \eqref{12}
holds for any positive integer smaller than $m+n$. For
$f\in{\Z}^{m+n}$, we write $f':=f_I$ and $f'':=f_{\{m+n\}}$, where
$I:=\lbr m+n-1\rbr$. Then $M_f=M_{f'}\otimes M_{f''}$. From
\eqref{quasi:R:matrix, k} and \eqref{quasi:R:matrix}, we have
$$\Theta=\sum_{\mu\in \Z_+\Pi}\Theta_{\mu}\qquad \text{and} \qquad \Theta_{\mu}=\sum_{i=1}^{\rm{dim}_{\Q(q)} \UU_{\mu}^+ }b_{\mu,i}'\otimes b''_{\mu,i}\in \UU_{\mu}^+\otimes \UU_{\mu}^-.$$
Therefore we have
\begin{eqnarray*}
    \langle M_f,M_g\rangle
        &=&(\sigma(\Theta(\overline{M_{g'}}\otimes\overline{M_{g''}})), M_{f'}\otimes  M_{f''})\\
  &=&\sum_{\mu,i}(\sigma(b'_{\mu,i}\overline{M_{g'}}), M_{f'})  (\sigma(b''_{\mu,i}\overline{M_{g''}}), M_{f''})
\end{eqnarray*}
By Lemma~\ref{th3}, we have that
\begin{equation*}\nonumber
  (\sigma(b'_{\mu,i}\overline{M_{g'}}), M_{f'})=
  (\tau(\overline{\sigma(b'_{\mu,i})})\sigma(\overline{M_{g'}}),M_{f'})=
  (\sigma(\overline{M_{g'}}),\overline{\sigma(b'_{\mu,i})}M_{f'})
\end{equation*} and
\begin{equation*}\nonumber
  (\sigma(b''_{\mu,i}\overline{M_{g''}}), M_{f''})=
  (\tau(\overline{\sigma(b''_{\mu,i})})\sigma(\overline{M_{g''}}), M_{f''})=
  (\sigma(\overline{M_{g''}}),\overline{\sigma(b''_{\mu,i})}M_{f''}).
\end{equation*}
Therefore we have

\begin{eqnarray}\nonumber
\langle M_f,M_g\rangle &=&\sum_{\mu,i}
(\sigma(\overline{M_{g'}}),\overline{\sigma(b'_{\mu,i})}M_{f'})
  (\sigma(\overline{M_{g''}}),\overline{\sigma(b''_{\mu,i})}M_{f''})\\
&=&\sum_{\mu,i}\langle\overline{\sigma(b'_{\mu,i})}M_{f'},M_{g'}\rangle
\langle\overline{\sigma(b''_{\mu,i})}M_{f''},M_{g''}\rangle\nonumber\\
&\stackrel{(1)}{=}&\sum_{\mu,i}\langle
M_{g'},\overline{\sigma(b'_{\mu,i})}M_{f'}\rangle \langle
M_{g''},\overline{\sigma(b''_{\mu,i})}M_{f''}\rangle\nonumber\\
&\stackrel{(2)}{=}&\sum_{\mu,i}(\sigma(\sigma(b'_{\mu,i})\overline{M_{f'}}), M_{g'})(\sigma(\sigma(b''_{\mu,i})\overline{M_{f''}}),M_{g''})\nonumber\\
&=&(\sigma(\sum_{\mu,i}\sigma(b'_{\mu,i})\overline{M_{f'}}\otimes
  \sigma(b''_{\mu,i})\overline{M_{f''}}), M_{g'}\otimes  M_{g''})\nonumber\\
&\stackrel{(3)}{=}&( \sigma(\Theta(\overline{M_{f'}}\otimes\overline{M_{f''}})), M_{g'}\otimes  M_{g''})\nonumber\\
&=&\langle M_g,M_f\rangle.\nonumber
\end{eqnarray}
We have used the induction on $m+n$, Proposition~\ref{X-bar-H} and
Proposition~\ref{sigma-theta} for equalities (1), (2) and (3), respectively.
\end{proof}

\begin{thm} Let $\{T_f|f\in\mathbb{Z}^{m+n}\}$ and
$\{L_f|f\in\mathbb{Z}^{m+n}\}$ are canonical and dual canonical
bases of $\widehat{\mathbb{T}}^b$ respectively. Then
  \begin{equation*}
    \langle L_f,T_{-g}\rangle=\delta_{f,g}\ \ \mbox{for}\ \mbox{any}\
    f,g\in\mathbb{Z}^{m+n}.
  \end{equation*}
\end{thm}

\begin{proof}
 We can adapt exactly the same proof of \cite[Theorem 2.23]{Br}.
\end{proof}

\begin{cor} \label{cor:el} We have the following formulas:
\begin{equation*}\label{aux:eq11}
\sum_{h}t_{hf}(q)\ell_{-h,-g}(q^{-1})=\delta_{f,g},\ \ \ \
\sum_{h}\ell_{hf}(q)t_{-h,-g}(q^{-1})=\delta_{f,g}.
\end{equation*}
\end{cor}

\begin{cor}\label{cor:inv}
For $f\in\mathbb{Z}^{m+n}$,
\begin{equation} \label{aux:eq12}
M_f=\sum_{h\preceq_{\bf b}f}t_{-f,-h}(q^{-1})L_h
 = \sum_{h\preceq_{\bf b}f}\ell_{-f,-h}(q^{-1})T_h.
\end{equation}
\end{cor}

\begin{proof}
  The proof is exactly as in \cite[Corollary 2.24]{Br}. Write $M_f=\sum_h \langle M_f,T_{-h}\rangle L_h$. Then the
  definition of $\langle\cdot,\cdot\rangle$ tells us that $\langle
  M_f,T_{-h}\rangle=t_{-f,-h}(q^{-1})$. The second equality is
  similar.
\end{proof}

\begin{rem}
An equivalent version of Corollary 3.6 has been obtained by Brundan, Losev and Webster via graded tensor product categorifications (see \cite[Section 5.9]{BLW} for details).
\end{rem}

In the remainder of this section we shall follow the definitions and conventions given in \cite{CLW2}.
For a  $0^m1^n$-sequence ${\bf b}$, we recall that $M_{\bf b}(\la)$,  $L_{\bf b}(\la)$ and $T_{\bf b}(\la)$ denote the Verma module, the highest weight irreducible module and the tilting module respectively for the integral weight $\la$ in the BGG category $\mc{O}^{m|n}_{\bf b}$ of modules over the general linear superalgebra $\gl(m|n)$ with respect to the Borel subalgebra $\mf{b}_{\bf b}$.
 A part of the Brundan-Kazhdan-Lusztig conjecture (BKL conjecture) \cite[Conjecture 8.1]{CLW2} states that
 \begin{align*}
(T_{\bf b}(\la):M_{\bf b}(\mu))=t_{f^{\bf b}_\mu f^{\bf b}_\la}(1),
\end{align*}
where $(T_{\bf b}(\la):M_{\bf b}(\mu))$ denotes the multiplicity of  the Verma module $M_{\bf b}(\mu)$ in the Verma flag of the tilting modules $T_{\bf b}(\la)$
and $f^{\bf b}_\gamma$  ( (see \cite[(6.9)]{CLW2} for the precise definition) is an element in $\Z^{m+n}$ determined by an integral weight $\gamma$.
By \cite[(6.9)]{CLW2} (cf.\cite[Theorem 6.4]{Br2}), the BKL conjecture implies
\begin{align}\label{tilt:verma:irred}
[M_{\bf b}(-\mu-2\rho_{\bf b}):L_{\bf b}(-\la-2\rho_{\bf b})]=\left(T_{\bf b}(\la):M_{\bf b}(\mu)\right)=t_{f^{\bf b}_\mu f^{\bf b}_\la}(1),
\end{align}
where $[M_{\bf b}(\gamma):L_{\bf b}(\eta)]$ denotes the multiplicity of  the irreducible module $L_{\bf b}(\eta)$ in the composition series  of the Verma module $M_{\bf b}(\gamma)$
and $\rho_{\bf b}$  ( (see \cite[(6.5)]{CLW2} for the precise definition) is an  integral weight depending on the $0^m1^n$-sequence ${\bf b}$.
Since every Verma module has a finite composition series,  every  integral weight tilting module has a finite Verma flag  by \cite[Proposition 3.9]{CLW2} and $t^{\bf b}_{gf}(q)\in\Z_+[q]$ (see Remark~\ref{positive}),  the validity of the BKL conjecture proved in \cite{CLW2, BLW} implies the following proposition.
\begin{prop}\label{prop:t-poly-finite}
Let ${\bf b}$ be a $0^m1^n$-sequence.
\begin{itemize}
  \item[(i)] For each $f\in\mathbb{Z}^{m+n}$, there are only finitely many $g\in\mathbb{Z}^{m+n}$ such that $t^{{\bf b}}_{gf}(q)\not=0$.
  \item[(ii)] For each $g\in\mathbb{Z}^{m+n}$, there are only finitely many $f\in\mathbb{Z}^{m+n}$ such that $t^{{\bf b}}_{gf}(q)\not=0$.
  \end{itemize}
 \end{prop}

 The following is an obvious consequence of the proposition above.

\begin{prop}\label{prop:can-finite}
Let ${\bf b}$ be a $0^m1^n$-sequence.
\begin{itemize}
  \item[(i)] For every $T_f^{{\bf b}}\in \widehat{\mathbb{T}}^{{\bf b}}$, we have $T_f^{{\bf b}}\in {\mathbb{T}}^{{\bf b}}$.
  \item[(ii)] For each $f\in\mathbb{Z}^{m+n}$, the expression of $M_f=\sum_{h\preceq_{\bf b}f}t_{-f,-h}(q^{-1})L_h$ in \eqnref{aux:eq12} is a finite sum.
Moreover, for each given $h\in\mathbb{Z}^{m+n}$ there are only finitely many $M_f$ in the  standard monomial basis such that $L_h$ appears in the  expression of $M_f$ in \eqnref{aux:eq12} with nonzero coefficients.
\end{itemize}
 \end{prop}

 \begin{rem} The formula \eqnref{tilt:verma:irred} and finiteness of  Verma flag  of every integral weight tilting module imply that there are only finitely many Verma modules  containing a given irreducible integral weight module $L_{\bf b}(\la)$  in their composition series.
 \end{rem}

\section{Inversion formula in $\widehat{\mathscr{E}}^{\bf b}$}
In this section, we define a symmetric bilinear form on $\widehat{ \mathscr{E}}^{{\bf b}}$ and showing that the canonical
basis $\{T^{\bf b}_f|\,f\in\mathbb{Z}^{{\bf b},+}\}$ and the dual canonical basis
$\{L^{\bf b}_f|\,f\in\mathbb{Z}^{{\bf b},+}\}$ are dual with respect to the symmetric bilinear form.

Recall that the symmetric  linear form $(\cdot,\cdot)$ on
$\mathbb{T}^{{\bf b}}$ is determined by $(M_f,M_g)=\delta_{f,g}$. For $f,g\in\mathbb{Z}^{{\bf b},+}$, we have
\begin{equation*}
(K_f, K_g)
=-(\prod_{k=1}^d[r_k-r_{k-1}]!)(M_{f\cdot w_0}, M_{g\cdot
w_0}H_0)=-(-q)^{-\ell(w_0)}(\prod_{k=1}^d[r_k-r_{k-1}]!)\delta_{f,g}
\end{equation*}
by Lemma~\ref{lem:H0} (iii, iv) and Lemma~\ref{th3} (i).
Define a symmetric bilinear form $(\cdot,\cdot)_{\mathscr{E}}$ on
$\mathscr{E}^{\bf b}$ by
$$
(K_f,K_g)_{\mathscr{E}}:=\frac{-1}{(-q)^{-\ell(w_0)}(\prod_{k=1}^d[r_k-r_{k-1}]!)}(K_f,K_g).
$$
Then $\{K_f|f\in\mathbb{Z}^{{\bf b},+}\}$ is  an orthonormal basis
of $\mathscr{E}^{\bf b}$.

By Lemma~\ref{th3}(ii) and Lemma~\ref{lem:MfH0}, the antilinear
involution $\sigma$ on $\mathbb{T}^{{\bf b}}$
 defined in \eqref{def:sigma on T} leaves $\mathscr{E}^{{\bf b}}$
invariant and
\begin{equation*}
  \sigma(K_f)=(-q)^{\ell(w_0)}K_{-f\cdot w_0},  \quad \text{for all } f\in\mathbb{Z}^{{\bf b},+}.
\end{equation*}
We  define the symmetric bilinear form $\langle\cdot,\cdot\rangle_{\mathscr{E}}$
on $\widehat{\mathscr{E}}^{{\bf b}}$ by
\begin{equation*}
  \langle u,v\rangle_{\mathscr{E}}:=\frac{-1}{\prod_{k=1}^d[r_k-r_{k-1}]!}\langle u,v\rangle, \quad \text{for all } u,v\in\widehat{ \mathscr{E}}^{{\bf b}}.
\end{equation*}
It is easy to see that the bilinear form $(\cdot,\cdot)_{\mathscr{E}}$ on
$\mathscr{E}^{\bf b}$ can be extended to a bilinear map $(\cdot,\cdot)_{\mathscr{E}}$ on
$\widetilde{\mathbb{T}}^{\bf b}H_0\times \mathscr{E}^{\bf b}$ and
\begin{equation*}
 \langle u,v\rangle_{\mathscr{E}}=(-q)^{-\ell(w_0)}\sum_{g\in \mathbb{Z}^{{\bf b},+}}(u,K_g)_{\mathscr{E}}( \sigma(\overline{v}),K_{g})_{\mathscr{E}},
 \qquad \forall\, u,v\in\widehat{ \mathscr{E}}^{{\bf b}}.
\end{equation*}

The following theorem is a direct generalization of \cite[Theorem
3.13]{Br}. Since the proof is similar, we omit it.

\begin{thm}
  For $f,g\in \mathbb{Z}^{{\bf b},+}$, $\langle L_f,\mc{U}_{-g\cdot
  w_0}\rangle_{\mathscr{E}}=\delta_{f,g}$.
\end{thm}

\begin{cor}
 For $f,g\in \mathbb{Z}^{{\bf b},+}$,
\begin{equation*}
\sum_{h\in\mathbb{Z}^{{\bf b},+}}u_{-h\cdot w_0,-f\cdot
w_0}(q)\ell_{h,g}(q^{-1})=\delta_{f,g},\ \ \ \
\sum_{h\in\mathbb{Z}^{{\bf b},+}}\ell_{-h\cdot w_0,f}(q)u_{h,-g\cdot
w_0}(q^{-1})=\delta_{f,g}.
\end{equation*}
\end{cor}

\begin{cor}
   For $f\in \mathbb{Z}^{{\bf b},+}$,
   \begin{equation*}
     K_f=\sum_{h\in \mathbb{Z}^{{\bf b},+}}u_{-f\cdot w_0,-h\cdot
     w_0}(q^{-1})L_h=\sum_{h\in \mathbb{Z}^{{\bf b},+}}\ell_{-f\cdot w_0,-h\cdot
     w_0}(q^{-1})\mc{U}_{h}.
   \end{equation*}
\end{cor}

\begin{rem}
For a $0^m1^n$-sequence ${\bf b}$, we recall that the tensor space $\mathbb T^{{\bf b},I_k}$ is defined in \eqnref{T:I} for $k\in\lbr d\rbr$.  The results in this section have a counterpart for $q$-wedge subspaces of $\mathbb T^{\bf b}= \mathbb T^{{\bf b},I_1}\otimes
\mathbb T^{{\bf b},I_2}\otimes\cdots\otimes \mathbb T^{{\bf
b},I_d}$ of the following forms:
$$
\wedge^{n_{11}}\VV^{b_1}\otimes\cdots\otimes\wedge^{n_{1{p_1}}}\VV^{b_1}
\otimes\wedge^{n_{21}}\VV^{b_2}\otimes\cdots\otimes\wedge^{n_{2{p_2}}}\VV^{b_2}\otimes\cdots\otimes
\wedge^{n_{d1}}\VV^{b_d}\otimes\cdots\otimes\wedge^{n_{d{p_d}}}\VV^{b_d},
$$
where $\sum_{j=1}^{p_k}n_{kj}=r_k-r_{k-1}$ for $k\in\lbr d\rbr$. We leave the formulations and the proof of analogous results to the reader.
\end{rem}

\section{An algorithm}
In this section we explain a method to compute the elements in
the canonical basis of $\widehat{\mathbb{T}}^{{\bf b}}$ from  knowing the precise
expressions of the elements in the canonical basis of
$\widehat{\mathbb{T}}^{{\bf b}_\text{st}}$ in terms of standard
monomial basis.  Recall that the $0^m1^n$-sequence ${\bf b}_{\mathrm{st}}:=(1,\cdots,1,0,\cdots,0)$ consisting of $n$ ${1}$'s
followed by $m$ ${0}$'s. Combining with the Brundan's algorithm for
computing  the elements in the canonical basis of
$\widehat{\mathbb{T}}^{{\bf b}_{\mathrm{st}}}$, we have an algorithm
computing the elements in the canonical basis of
$\widehat{\mathbb{T}}^{{\bf b}}$ for arbitrary $0^m1^n$-sequence ${\bf b}$.

Recall that the $ \UU_k$-modules $\mathbb{V}_{k}$, $\mathbb{W}_{k}$
and  $\mathbb{T}^{\bf b}_{\leq|k|}$ are defined in
Section~\ref{subsection:Fock space}.
 For $r, k\in \mathbb{N}$, we recall that $\Z^{r}_k$ denotes the subset of $\Z^{r}$ consisting functions with values between $-k$ and $k$, and $\{M^{\bf b}_f \,|\, f\in\Z^{m+n}_k\}$ forms a basis of $\mathbb{T}^{{\bf b}}_{\leq|k|}$. The antilinear involution $\psi^{(k)}=\pi_k\circ\psi$ on $\mathbb{T}^{{\bf b}}_{\leq|k|}$ defined in Section~\ref{subsection:canonical} satisfies
    $$\psi^{(k)}(XM)=\overline{X}\psi^{(k)}(M),\qquad\text{for all}\,\, M\in\mathbb{T}^{{\bf b}}_{\leq|k|}\, \text{and}\, X\in \UU_k.$$
For each element $T^{{\bf b}}_f$ in the canonical basis of
$\widehat{\mathbb{T}}^{{\bf b}}$ with $f\in\Z^{m+n}_k$, we define
$$T^{{\bf b},k}_{f}:=\pi_k( T^{{\bf b}}_f)\in \mathbb{T}^{{\bf b}}_{\leq|k|}.$$
By \cite[Lemma 3.4]{CLW2}, we have
\begin{equation*}
\psi^{(k)}(T^{{\bf b},k}_{f})=T^{{\bf b},k}_{f},\qquad\text{for
all}\,\, f\in\Z^{m+n}_k,
\end{equation*}
and $\{T^{{\bf b},k}_f|f\in \Z^{m+n}_k\}$ forms a basis of $\mathbb{T}^{\bf b}_{\leq|k|}$.

 For
$f\in \Z^{1+1}$ with $f(1)=f(2)$, we define $f^{\downarrow}$ and $f^{\uparrow}$ in $\Z^{1+1}$ by
\begin{align}\label{eq:updown}
\begin{split}
f^{\downarrow}(1)=f^{\downarrow}(2)=f(1)-1,\\
f^{\uparrow}(1)=f^{\uparrow}(2)=f(1)+1.
\end{split}
\end{align}

The space $\widehat{\mathbb{T}}^{(0,1)}$ has the canonical basis
$\{T^{(0,1)}_f\,|\, f\in\Z^{1+1}\}$.   In this case, we have the
following formula \cite[Lemma~5.1]{CLW2}( cf.
\cite[Example~2.19]{Br}) for the elements in the the canonical
basis:
\begin{align*}
T^{(0,1)}_f=\begin{cases}
 M^{(0,1)}_f+q M^{(0,1)}_{f^\downarrow}&\text{ if }f(1)=f(2),
  \\
 M^{(0,1)}_f&\text{ if }f(1)\not=f(2).
\end{cases}
\end{align*}
Similarly, $\widehat{\mathbb{T}}^{(1,0)}$ has the canonical basis
$\{T^{(1,0)}_f\,|\, f\in\Z^{1+1}\}$ such that
\begin{align*}
T^{(1,0)}_f=\begin{cases}
 M^{(1,0)}_f+q M^{(1,0)}_{f^\uparrow}&\text{ if }f(1)=f(2),
  \\
 M^{(1,0)}_f&\text{ if }f(1)\not=f(2).
\end{cases}
\end{align*}
Now we consider the truncated Fock spaces $\mathbb{T}^{(0,
1)}_{\leq|k|}=\mathbb{V}_{k}\otimes \mathbb{W}_{k}$ and
$\mathbb{T}^{(1,0)}_{\leq|k|}=\mathbb{W}_{k}\otimes \mathbb{V}_{k}$.
Let $M_f:=M^{(0,1)}_f=v_{f(1)}\otimes w_{f(2)}$
and  $M'_f:=M^{(1,0)}_f=w_{f(1)}\otimes v_{f(2)}$ for each $f\in\Z^{1+1}$. Define
\begin{align*}
    T^k_{f}&:=T^{(0,1)}_{f},
      \quad \mbox{if}\ (f(1),f(2))\not=(-k,-k),\\
  T'^k_f&:=T^{(1,0)}_{f},    \quad  \mbox{if}\ (f(1),f(2))\not=(k,k).
\end{align*}

  Then $\{M_f\,|\, f\in\Z^{1+1}_k\}$ and $\{T^k_{f}\,|\, f\in\Z^{1+1}_k,   (f(1),f(2))\not=(-k,-k)\}\cup\{M_{(-k,-k)}\}$ are bases for $\mathbb{V}_{k}\otimes \mathbb{W}_{k}$, and  $\{M'_f\,|\, f\in\Z^{1+1}_k\}$ and $\{T'^k_{f}\,|\, f\in\Z^{1+1}_k,  (f(1),f(2))\not=(k,k)\}\cup\{M'_{(k,k)}\}$ are bases for $\mathbb{W}_{k}\otimes \mathbb{V}_{k}$.

Let $ \mathbb{U}_k$ and $\mathbb{U}'_k$ be subspaces of
$\UU_k$-modules $\mathbb{T}^{(0, 1)}_{\leq|k|}$ and $\mathbb{T}^{(1,
0)}_{\leq|k|}$ spanned by $\{T^k_{f}\,|\,f\in\mathbb{Z}^{1+1}_k,\,
(f(1),f(2))\neq(-k,-k)\}$ and
$\{T'^k_{f}\,|\,f\in\mathbb{Z}^{1+1}_k,\, (f(1),f(2))\neq(k,k)\}$,
respectively. It is easy to check $\mathbb{U}_k$ and $\mathbb{U}'_k$
are $\UU_k$-submodules of $\mathbb{T}^{(0, 1)}_{\leq|k|}$ and
$\mathbb{T}^{(1, 0)}_{\leq|k|}$, respectively.

 Summarizing the above, we have the following.

\begin{lem}\label{canonical k}
$\mathbb{U}_k$ and $\mathbb{U}'_k$ are $\UU_k$-submodules of
$\mathbb{T}^{(0, 1)}_{\leq|k|}$ and $\mathbb{T}^{(1, 0)}_{\leq|k|}$,
respectively. $\{T^k_{f}\,|\,f\in\mathbb{Z}^{1+1}_k,\,
(f(1),f(2))\neq(-k,-k)\}$ and
$\{T'^k_{f}\,|\,f\in\mathbb{Z}_k^{1+1},\, (f(1),f(2))\neq(k,k)\}$
are  bases consisting of $\psi^{(k)}$-invariant  elements of
$\mathbb{U}_k$ and $\mathbb{U}'_k$, respectively.
\end{lem}

\begin{prop}\label{lem:Rk}
There is a $\UU_k$-module isomorphism
$\mathcal R^{(k)}:\mathbb{U}'_k\longrightarrow\mathbb{U}_k$
determined by
\begin{equation*}
{\mathcal{R}}^{(k)}(T'^k_{f})=
    \begin{cases}
   T^k_{f^\uparrow} & \mbox{if}\ f(1)=f(2)\neq k,\\
  T^k_{f\cdot\tau}& \mbox{if}\ f(1)\neq f(2),
  \end{cases}
\end{equation*}
where $f\cdot\tau\in \Z^{1+1}$ is defined by $f\cdot\tau(1)=f(2)$
and $f\cdot\tau(2)=f(1)$. Moreover, ${\mathcal{R}}^{(k)}$
commutes with $\psi^{(k)}$.
\end{prop}

\begin{proof}
Let $P: \mathbb{W}_{k}\otimes \mathbb{V}_{k} \longrightarrow
\mathbb{V}_{k}\otimes \mathbb{W}_{k}$ be the linear map defined by
$P(w\otimes v)=v\otimes w$ for $w\in\mathbb{W}_{k}$ and $v\in
\mathbb{V}_{k}$ and let $\tilde{f}: \mathbb{V}_{k}\otimes
\mathbb{W}_{k} \longrightarrow \mathbb{V}_{k}\otimes \mathbb{W}_{k}$
be the linear map defined by $\tilde{f}(v_i\otimes
w_j)=q^{-\delta_{i,j}}v_i\otimes w_j$ for $-k\leq i,j\leq k$. By
\cite[Theorem 7.3]{Jan}, $\Theta^{(k)}\circ \tilde{f}\circ P:
\mathbb{W}_{k}\otimes \mathbb{V}_{k} \longrightarrow
\mathbb{V}_{k}\otimes \mathbb{W}_{k}$ is a $\UU_k$-module
isomorphism. Note that $E_a$, $F_a$ and $q$ are replaced by $F_a$,
$E_a$ and $q^{-1}$ in {\em loc.~cit.}, respectively. For $f\in
\mathbb{Z}_k^{1+1}$ with $f(1)=f(2)\neq k$, we have
$$
\Theta^{(k)}\circ\tilde{f}\circ P(  T'^k_{f})=\Theta^{(k)}(q^{-1}
M_f+M_{f^{\uparrow}})=\psi^{(k)}(q
M_f+M_{f^{\uparrow}})=T^k_{f^{\uparrow}}.
$$
Similarly, we have $\Theta^{(k)}\circ\tilde{f}\circ P(T'^k_{f})=
T^k_{f\cdot\tau}$ for $f\in \mathbb{Z}_k^{1+1}$ with $f(1)\neq
f(2)$. Therefore the restriction  of the $\UU_k$-module isomorphism
$\Theta^{(k)}\circ\tilde{f}\circ P$  to $ \mathbb{U}'_k$ gives the
$\UU_k$-module  isomorphism ${\mathcal{R}}^{(k)}$ satisfying the
required property. ${\mathcal{R}}^{(k)}$ commuting with
$\psi^{(k)}$ follows from Lemma~\ref{canonical k} that $T^k_{f}$'s
and $T'^k_{f}$'s are $\psi^{(k)}$-invariant  elements.
\end{proof}

For a $0^m1^n$-sequence ${\bf b}=(b_1,b_2, \cdots,b_{m+n})$ with $b_\kappa=0$ and
$b_{\kappa+1}=1$, the $0^m1^n$-sequence $s_\kappa{\bf b}$ is defined
by
$$
s_\kappa{\bf b}:=(b_1,b_2, \cdots,b_{\kappa-1},1,0, b_{\kappa+2},\cdots b_{m+n}).
$$
That is  the $0^m1^n$-sequence $s_\kappa{\bf b}$ is obtained from
${\bf b}$ by switching the $\kappa$-coordinate and
$(\kappa+1)$-coordinate of ${\bf b}$.

For the $0^m1^n$ sequence ${\bf b}=({\bf b}^1,0,1,{\bf b}^2)$ wtih
$0^{m_1}1^{n_1}$-sequence ${\bf b}^1$ and $0^{m_2}1^{n_2}$-sequence
${\bf b}^2$ satisfying $m=m_1+m_2+1$ and $n=n_1+n_2+1$,   we define
subspaces
\begin{equation*}
  \mathbb{U}^{{\bf b},\kappa}_{k}:=\mathbb{T}^{{\bf b}^1}_{\leq |k|}\otimes\mathbb{U}_k\otimes
  \mathbb{T}^{{\bf b}^2}_{\leq |k|}\subseteq\mathbb{T}^{{\bf b}}_{\leq |k|}
  \quad \ \mbox{and}\quad
  \mathbb{U}'^{{\bf b},\kappa}_{k}:=\mathbb{T}^{{\bf b}^1}_{\leq |k|}\otimes\mathbb{U}'_k\otimes
  \mathbb{T}^{{\bf b}^2}_{\leq |k|}\subseteq\mathbb{T}^{s_\kappa{\bf b}}_{\leq |k|},
\end{equation*}
where $\kappa:=m_1+n_1+1$. Since $\mathbb{U}_k$ and $\mathbb{U}'_k$
are $\psi^{(k)}$-invariant $\UU_k$-modules, $\mathbb{U}^{{\bf
b},\kappa}_{k}$ and $\mathbb{U}'^{{\bf b},\kappa}_{k}$
 are $\psi^{(k)}$-invariant $\UU_k$-submodules of $\mathbb{T}^{{\bf b}}_{\leq |k|}$ and
 $\mathbb{T}^{s_\kappa{\bf b}}_{\leq |k|}$, respectively. For a $0^m1^n$-sequence  ${\bf b}$ with $b_\kappa=0$ and $b_{\kappa+1}=1$ and $f\in\Z^{m+n}$, we define\begin{align}\label{def:U:ka}
  U^{{\bf b},\kappa}_{f}&:=\texttt{v}^{b_1}_{f(1)}\otimes\cdots\otimes
  \texttt{v}^{b_{\kappa-1}}_{f(\kappa-1)}\otimes
  T_{(f(\kappa),f(\kappa+1))}\otimes
  \texttt{v}_{f(\kappa+2)}^{b_{\kappa+2}}\otimes\cdots\otimes
  \texttt{v}_{f(m+n)}^{b_{m+n}},\\\label{def:U':ka}
U'^{{\bf b},\kappa}_{f}&:=\texttt{v}^{b_1}_{f(1)}\otimes\cdots\otimes
  \texttt{v}^{b_{\kappa-1}}_{f(\kappa-1)}\otimes
  T'_{(f(\kappa),f(\kappa+1))}\otimes
  \texttt{v}_{f(\kappa+2)}^{b_{\kappa+2}}\otimes\cdots\otimes
  \texttt{v}_{f(m+n)}^{b_{m+n}}.
       \end{align}
Recall that $T^{k}_{(f(\kappa),f(\kappa+1))}$  and
$T'^{k}_{(f(\kappa),f(\kappa+1))}$ are $\psi^{(k)}$-invariant
elements in $\mathbb{U}_k$ and $\mathbb{U}'_k$, respectively. Note
that $\{ U^{{\bf b},\kappa}_{f}\,|\,f\in\mathbb{Z}^{m+n}_k,\,
(f(\kappa),f(\kappa+1))\neq(-k,-k)\}$ and $\{ U'^{{\bf
b},\kappa}_{f}\,|\,f\in\mathbb{Z}^{m+n}_k,\,
(f(\kappa),f(\kappa+1))\neq(k,k)\}$ are bases of $\mathbb{U}^{{\bf
b},\kappa}_{k}$ and $\mathbb{U}'^{{\bf b},\kappa}_{k}$,
respectively.

  The following proposition follows from the characterization of canonical basis \cite[Theorem 27.3.2]{Lu} and the $\psi^{(k)}$-invariant property of the $\UU_k$-modules $\mathbb{U}_k$ and $\mathbb{U}'_k$ (cf. \cite[Lemma 5.7, Proposition 5.8]{CLW2}).

\begin{prop}\label{prop:canU} Let ${\bf b}=(b_1,b_2,\ldots,b_{m+n})$ be a $0^m1^n$-sequence with $b_\kappa=0$ and $b_{\kappa+1}=1$.
\begin{itemize}
  \item[(i)] For $f\in\Z^{m+n}_k$ with $(f(\kappa), f(\kappa+1))\not= (-k,-k)$, we have $T_f^{{\bf b},k}\in \mathbb{U}^{{\bf b},\kappa}_{k}$ and
  $$
  T_f^{{\bf b},k}=U^{{\bf b},\kappa}_{f}+\sum_{h\preceq_{\bf b}f,\, h\in\Z^{m+n}_k} \alpha_{hf} U^{{\bf b},\kappa}_{h}, \quad\mbox{where}\,\,\alpha_{hf}\in q\Z[q].
  $$
  In particular, $\{ T^{{\bf b},k}_{f}\,|\,f\in\mathbb{Z}^{m+n}_k,\,(f(\kappa),f(\kappa+1))\neq(-k,-k)\}$
  is basis of $\mathbb{U}^{{\bf b},\kappa}_{k}$.
  \item[(ii)] For $f\in\Z^{m+n}_k$ with $(f(\kappa), f(\kappa+1))\not= (k,k)$, we have $T_f^{s_\kappa{\bf b},k}\in \mathbb{U}'^{{\bf b},\kappa}_{k}$ and
  $$
T_f^{s_\kappa{\bf b},k}=U'^{{\bf b},\kappa}_{f}+\sum_{h\preceq_{s_\kappa{\bf b}}f,\, h\in\Z^{m+n}_k} \alpha'_{hf} U'^{{\bf b},\kappa}_{h}, \quad\mbox{where}\,\,\alpha'_{hf}\in q\Z[q].
  $$
  In particular, $\{ T_f^{s_\kappa{\bf b},k}\,|\,f\in\mathbb{Z}^{m+n}_k,\,(f(\kappa),f(\kappa+1))\neq(k,k)\}$
  is basis of $\mathbb{U}'^{{\bf b},\kappa}_{k}$.
\end{itemize}
 \end{prop}

For $f\in\mathbb{Z}^{m+n}$ and $\kappa\in \lbr m+n-1\rbr$ with
$f(\kappa)=f(\kappa+1)$, we define $f^{\uparrow, \kappa}\in
\mathbb{Z}^{m+n}$ by
  \begin{equation*}
f^{\uparrow, \kappa}(j):=\left\{
                            \begin{array}{ll}
                              f(j), & \hbox{if \,$j\not=\kappa,\kappa+1$,} \\
                                                  f^{\uparrow}(\kappa)+1, & \hbox{if \,$j=\kappa,\kappa+1$.}
                            \end{array}
                          \right.
  \end{equation*}
  For $f\in\mathbb{Z}^{m+n}$ and $\kappa\in \lbr m+n-1\rbr$ with
$f(\kappa)\not=f(\kappa+1)$, we define $f\cdot\tau_\kappa\in \Z^{m+n}$ by
  \begin{equation*}
f\cdot\tau_\kappa(j):=\left\{
                            \begin{array}{ll}
                              f(j), & \hbox{if \,$j\not=\kappa,\kappa+1$,} \\
                                                  f(\kappa+1), & \hbox{if \,$j=\kappa$,}\\
                                                  f(\kappa), & \hbox{if \,$j=\kappa+1$.}
                            \end{array}
                          \right.
  \end{equation*}
For the $0^m1^n$ sequence ${\bf b}=({\bf b}^1,0,1,{\bf b}^2)$  such
that ${\bf b}^1$ and ${\bf b}^2$ are respectively
$0^{m_1}1^{n_1}$-sequence and $0^{m_2}1^{n_2}$-sequence,   we let $
{\mathcal{R}}^{(k)}_\kappa:=1^{{\bf
b}^1}_{k}\otimes{\mathcal{R}}^{(k)}\otimes1^{{\bf b}^2}_{k}$,
where $\kappa:=m_1+n_1+1$, $1^{{\bf b}^1}_{k}$ and $1^{{\bf
b}^2}_{k}$ are identity maps on ${\mathbb{T}^{{\bf b}^1}_{\leq|k|}}$
and ${\mathbb{T}^{{\bf b}^2}_{\leq|k|}}$, respectively. The linear
map
$$
 {\mathcal{R}}^{(k)}_\kappa\, :\, \mathbb{U}'^{{\bf b},\kappa}_{k}\longrightarrow\mathbb{U}^{{\bf b},\kappa}_{k}
$$
 is an isomorphism of  $\UU_k$-modules since $1^{{\bf b}^1}_{k}$, ${\mathcal{R}}^{(k)}$ and $1^{{\bf b}^2}_{k}$ are  isomorphisms of  $\UU_k$-modules. By Proposition~\ref{lem:Rk}, we have
    \begin{equation}\label{U}
{\mathcal{R}}^{(k)}_{\kappa}(U'^{{\bf b},\kappa}_{f})=
  U^{{\bf b},\kappa}_{s_\kappa{f} }\quad \hbox{for $f\in\Z^{m+n}_k$ with $(f(\kappa), f(\kappa+1))\not= (k,k)$,}
\end{equation}
where $s_\kappa : \Z^{m+n}\longrightarrow \Z^{m+n}$ is the bijection defined by
 \begin{equation}\label{sf}
s_\kappa{f}:=
    \begin{cases}
f^{\uparrow, \kappa} & \mbox{if}\ f(\kappa)=f(\kappa+1),\\
f\cdot\tau_\kappa & \mbox{if}\ f(\kappa)\not=f(\kappa+1).
  \end{cases}
\end{equation}
 Since $\mathbb{U}^{{\bf
b},\kappa}_{k}$ and $\mathbb{U}'^{{\bf b},\kappa}_{k}$
 are $\psi^{(k)}$-invariant $\UU_k$-submodules of $\mathbb{T}^{{\bf b}}_{\leq |k|}$ and
 $\mathbb{T}^{s_\kappa{\bf b}}_{\leq |k|}$, and the isomorphisms $1^{{\bf b}^1}_{k}$,
 ${\mathcal{R}}^{(k)}$ and $1^{{\bf b}^2}_{k}$ commute with $\psi^{(k)}$, we have
 ${\mathcal{R}}^{(k)}_\kappa$ commuting with $\psi^{(k)}$.

Summarizing the above, we have the first part of the following.

\begin{prop}  \label{tilde(R)} For a $0^m1^n$-sequence ${\bf b}=(b_1,b_2,\ldots,b_{m+n})$  with $b_\kappa=0$ and $b_{\kappa+1}=1$,
the linear map
$$
 {\mathcal{R}}^{(k)}_\kappa\, :\, \mathbb{U}'^{{\bf b},\kappa}_{k}\longrightarrow\mathbb{U}^{{\bf b},\kappa}_{k}
$$
 is an isomorphism of  $\UU_k$-modules such that
     \begin{equation*}
{\mathcal{R}}^{(k)}_{\kappa}(U'^{{\bf b},\kappa}_{f})=
  U^{{\bf b},\kappa}_{s_\kappa{f} }\quad \hbox{for $f\in\Z^{m+n}_k$ with $(f(\kappa), f(\kappa+1))\not= (k,k)$}.
\end{equation*}
 Also $ {\mathcal{R}}^{(k)}_\kappa$ commutes with $\psi^{(k)}$. Moreover,  we have
\begin{equation}\label{T:k}
{\mathcal{R}}^{(k)}_{\kappa}(T^{s_\kappa{\bf b},k}_{f})=  T^{{\bf b},k}_{s_\kappa f}, \quad \hbox{ for }\, f\in\Z^{m+n}_k\,\, \hbox{ with }\, (f(\kappa),f(\kappa+1))\neq (k,k),
\end{equation}
and
\begin{equation}\label{T}
{\mathcal{R}}^{(k)}_{\kappa}(T^{s_\kappa{\bf b}}_{f})=  T^{{\bf b}}_{s_\kappa f}, \quad \hbox{ if }\, T^{s_\kappa{\bf b}}_{f}\in
 \mathbb{T}^{s_\kappa{\bf b}}_{\leq |k|}.
\end{equation}
\end{prop}

\begin{proof}
First we show \eqnref{T:k} holds. For $f\in\Z^{m+n}_k$ with $f(\kappa)=f(\kappa+1)\neq k$, we have
\begin{equation}\label{RT}
{\mathcal{R}}^{(k)}_{\kappa}(T^{s_\kappa{\bf b},k}_{f})=U^{{\bf b},\kappa}_{f}+\sum_{ h\in\Z^{m+n}_k} \alpha'_{hf} U^{{\bf b},\kappa}_{h}, \quad\mbox{where}\,\,\alpha'_{hf}\in q\Z[q],
\end{equation}
by Proposition~\ref{prop:canU}. Since $\{ T^{{\bf b},k}_{f}\,|\,f\in\mathbb{Z}^{m+n}_k,\,(f(\kappa),f(\kappa+1))\neq(-k,-k)\}$ is a basis of
$\mathbb{U}^{{\bf b},\kappa}_{k}$,  we can write ${\mathcal{R}}^{(k)}_{\kappa}(T^{s_\kappa{\bf b},k}_{f})=\sum_{g\in \Z^{m+n}_k} n_g  T^{{\bf b},k}_{g}$
 with $n_g\in \Q(q)$. By \eqnref{RT} and  Proposition~\ref{prop:canU}, we have $n_g\in q\Z[q]$ for all $g\in \Z^{m+n}_k$ with $g\not={f^{\uparrow, \kappa} } $
 and $n_{f^{\uparrow, \kappa} }\in 1+q\Z[q] $. By the first part of the proposition, ${\mathcal{R}}^{(k)}_{\kappa}(T^{s_\kappa{\bf b},k}_{f})$ is a $\psi^{(k)}$-invariant element in $\mathbb{U}^{{\bf b},\kappa}_{k}\subseteq \mathbb{T}^{{\bf b}}_{\leq |k|}$. Therefore
 $n_g$ are contained in $\Z$ for all $g\in \Z^{m+n}_k$.
Hence ${\mathcal{R}}^{(k)}_{\kappa}(T^{s_\kappa{\bf b},k}_{f})= T^{{\bf b},k}_{f^{\uparrow, \kappa} }$.
The case for $f\in\Z^{m+n}_k$ with $ f(\kappa)\not=f(\kappa+1)$ is similar.

For $T^{s_\kappa{\bf b}}_{f}\in  \mathbb{T}^{s_\kappa{\bf b}}_{\leq |k|}$, we have $T^{s_\kappa{\bf b}}_{f}=T^{s_\kappa{\bf b},k}_{f}$. By Proposition~\ref{prop:can-finite}, we may assume that $T^{{\bf b}}_{s_\kappa f}=T^{{\bf b},l}_{s_\kappa f}$ for some $l\ge k$. By Proposition~\ref{prop:canU},  we have $T^{{\bf b}}_{s_\kappa f}\in \mathbb{U}^{{\bf b},\kappa}_{k}$ and
$$
({\mathcal{R}}^{(l)}_{\kappa})^{-1}(T^{{\bf b},l}_{s_\kappa f})
=T^{s_\kappa{\bf b},l}_{f}=T^{s_\kappa{\bf b},k}_{f}.
$$
This implies $T^{{\bf b},l}_{s_\kappa f}\in \mathbb{T}^{s_\kappa{\bf b}}_{\leq |k|}$. Hence
${\mathcal{R}}^{(k)}_{\kappa}(T^{s_\kappa{\bf b}}_{f})= T^{{\bf b},l}_{s_\kappa f} =T^{{\bf b}}_{s_\kappa f}$.
 \end{proof}

Now we explain a method to compute the elements in the canonical
basis of $\widehat{\mathbb{T}}^{{\bf b}}$ from  knowing the precise
expressions of the elements in the canonical basis of
$\widehat{\mathbb{T}}^{{\bf b}_\text{st}}$ in terms of standard
monomial basis.
 Recall that ${\bf b}_{\mathrm{st}}$ is the  $0^m1^n$-sequence
$(1,\ldots,1,0,\ldots,0)$. For ${\bf
b}\not={\bf b}_{\mathrm{st}}$, let $
1\leq i_1<i_2<\ldots<i_n\leq m+n $  be integers such that
$b_{i_1}=b_{i_2}=\cdots=b_{i_n}=1$ and let $l$ be the smallest
number in $\lbr n\rbr$ such that $i_l\not=l$. We have
\begin{equation}\label{sb}
 {\bf b}_{\mathrm{st}}=s_{n}s_{n+1}\cdots s_{i_n-1}\cdots
 s_{l+1}s_{l+2}\cdots  s_{i_{l+1}-1}s_{l}s_{l+1}\cdots  s_{i_{l}-1} {\bf b}.
\end{equation}
Now we want to compute $T_f^{{\bf b}}$. Let
$$
g=(s_{i_{l}-1}\cdots s_{l+1}s_{l}\cdots s_{i_{n-1}-1}\cdots
  s_{n}s_{n-1}s_{i_n-1}\cdots s_{n+1}s_{n})^{-1}(f)\in\Z^{m+n}.
$$
Recall that $s_\kappa h$ is defined for any $h\in\Z^{m+n}$ in \eqnref{sf}. By Proposition~\ref{prop:can-finite}, we may assume that
 $T_g^{{\bf  b}_{\mathrm{st}},k}\in  {\mathbb{T}}^{{\bf b}_{\mathrm{st}}}_{|k|}$ for some positive integer $k$. We shall drop the superscripts of ${\mathcal{R}}^{(k)}_{\kappa}$'s. By Proposition~ \ref{tilde(R)}, we have
 $$T_f^{{\bf b}}={\mathcal{R}}_{i_{l}-1}\cdots{\mathcal{R}}_{l+1}{\mathcal{R}}_l
\cdots{\mathcal{R}}_{i_{n-1}-1}\cdots
{\mathcal{R}}_n{\mathcal{R}}_{n-1}{\mathcal{R}}_{i_n-1}\cdots
{\mathcal{R}}_{n+1}{\mathcal{R}}_n
(T^{\bf{b}_{\mathrm{st}}}_g).$$
Therefore we can compute $T_f^{{\bf b}}$ step by step applying ${\mathcal{R}}_\kappa$ for some $\kappa$ to an element, say $T^{{\bf c}}_h$, in the canonical basis obtained from the previous step. The element $T^{{\bf c}}_h$ is contained in $\mathbb{U}'^{s^{-1}_\kappa{\bf c},\kappa}_{k}$ by Proposition~\ref{prop:canU}. Since $T^{{\bf c}}_h$ can be written as a linear combination of the elements of the form $U'^{s^{-1}_\kappa{\bf c},\kappa}_{h}$'s (recall that  $U'^{s^{-1}_\kappa{\bf c},\kappa}_{h}$ is defined in \eqnref{def:U':ka}),  $ T_{s_\kappa h}^{s^{-1}_\kappa{\bf c}}={\mathcal{R}}_\kappa(T_h^{{\bf c}})$ can be computed easily by Proposition~\ref{tilde(R)}. Repeating this processes, we have a procedure to compute $T_f^{{\bf b}}$.

Now we give an  example to  illustrate the  algorithm.
\begin{example} This example is based on the example given by Brundan \cite[Example 2.27]{Br}.
Let  ${\bf b}=(1,0,1,0,1,0)$ and
$f=(0,0,4,2,1,3)\in\mathbb{Z}^{3+3}$. Recall that $f\in\Z^{r}$ is identified with the
$r$-tuple $\left(f(1),f(2),\ldots,f(r)\right)$. We want to compute $T^{\bf b}_f$. By \eqnref{sb},  we have  ${\bf b}_{\mathrm{st}}=s_3s_4s_2{\bf b}$, where ${\bf b}_{\mathrm{st}}=(1,1,1,0,0,0)$.
Let $g=(s_2s_4s_3)^{-1}(f)$. We have
$$
g=s_3^{-1}s_4^{-1}s_2^{-1}(f)=s_3^{-1}s_4^{-1}(0,4,0,2,1,3)=s_3^{-1}(0,4,0,1,2,3)=(0,4,1,0,2,3).
$$
$ T^{{\bf b}_{\mathrm{st}}}_g$ has been computed in  \cite[Example 2.27]{Br} and
\begin{equation*}\begin{split}
  T^{{\bf b}_{\mathrm{st}}}_g=& M^{{\bf b}_{\mathrm{st}}}_{(0,4,1,0,2,3)}+qM^{{\bf b}_{\mathrm{st}}}_{(4,0,1,0,2,3)}+q(M^{{\bf b}_{\mathrm{st}}}_{(1,4,0,0,2,3)}+qM^{{\bf b}_{\mathrm{st}}}_{(1,4,1,1,2,3)})\\
  &+q^2(M^{{\bf b}_{\mathrm{st}}}_{(4,1,0,0,2,3)}+qM^{{\bf b}_{\mathrm{st}}}_{(4,1,1,1,2,3)}).
\end{split}\end{equation*}
Therefore
$ T^{{\bf b}_{\mathrm{st}}}_g\in \mathbb{T}_{\leq |4|}^{{\bf b}_{\mathrm{st}}}$ and $T^{\bf
b}_f={\mathcal{R}}^{(4)}_2{\mathcal{R}}^{(4)}_4{\mathcal{R}}^{(4)}_3(T^{{\bf b}_{\mathrm{st}}}_g)$. Let $\bf{c}=s^{-1}_3{\bf b}_{\mathrm{st}}=(1,1,0,1,0,0)$. Note that $ T^{{\bf b}_{\mathrm{st}}}_g$ has been written as  a linear combination of the elements of the form $U'^{{\bf c},3}_{h}$'s. Then
\begin{equation*}\begin{split}
T^{\bf c}_{s_3 g}=&{\mathcal{R}}^{(4)}_3(T^{{\bf b}_{\mathrm{st}}}_g)\\
=& M^{{\bf c}}_{(0,4,0,1,2,3)}+qM^{{\bf c}}_{(4,0,0,1,2,3)}+q(qM^{{\bf c}}_{(1,4,0,0,2,3)}+M^{{\bf c}}_{(1,4,1,1,2,3)})\\
  &+q^2(qM^{{\bf c}}_{(4,1,0,0,2,3)}+M^{{\bf c}}_{(4,1,1,1,2,3)})\\
=& M^{{\bf c}}_{(0,4,0,1,2,3)}+qM^{{\bf c}}_{(4,0,0,1,2,3)}+q^2M^{{\bf c}}_{(1,4,0,0,2,3)}+qM^{{\bf c}}_{(1,4,1,1,2,3)}\\
  &+q^3M^{{\bf c}}_{(4,1,0,0,2,3)}+q^2M^{{\bf c}}_{(4,1,1,1,2,3)} .
\end{split}\end{equation*}
Now we let ${\bf c}_1=s_4^{-1}{\bf c}=(1,1,0,0,1,0)$. Note that $T^{\bf c}_{s_3 g}$ has been written as  a linear combination of the elements of the form $U'^{{\bf c}_1,4}_{h}$'s. We have
\begin{equation*}\begin{split}
T^{{\bf c}_1}_{s_4 s_3 g}=&{\mathcal{R}}^{(4)}_4{\mathcal{R}}^{(4)}_3(T^{{\bf b}_{\mathrm{st}}}_g)\\
=& M^{{\bf c}_1}_{(0,4,0,2,1,3)}+qM^{{\bf c}_1}_{(4,0,0,2,1,3)}+q^2M^{{\bf c}_1}_{(1,4,0,2,0,3)}+qM^{{\bf c}_1}_{(1,4,1,2,1,3)}\\
  &+q^3M^{{\bf c}_1}_{(4,1,0,2,0,3)}+q^2M^{{\bf c}_1}_{(4,1,1,2,1,3)} \\
  =& M^{{\bf c}_1}_{(0,4,0,2,1,3)}+q^2M^{{\bf c}_1}_{(1,4,0,2,0,3)}+qM^{{\bf c}_1}_{(1,4,1,2,1,3)}+q^3M^{{\bf c}_1}_{(4,1,0,2,0,3)}\\
  &+q(M^{{\bf c}_1}_{(4,0,0,2,1,3)}+qM^{{\bf c}_1}_{(4,1,1,2,1,3)} ).
\end{split}\end{equation*}
Similarly, we have
\begin{equation*}\begin{split}
T^{{\bf b}}_f=&{\mathcal{R}}^{(4)}_2{\mathcal{R}}^{(4)}_4{\mathcal{R}}^{(4)}_3(T^{{\bf b}_{\mathrm{st}}}_g)\\
  =& M^{{\bf b}}_{(0,0,4,2,1,3)}+q^2M^{{\bf b}}_{(1,0,4,2,0,3)}+qM^{{\bf b}}_{(1,1,4,2,1,3)}+q^3M^{{\bf b}}_{(4,0,1,2,0,3)}\\
  &+q(qM^{{\bf b}}_{(4,0,0,2,1,3)}+M^{{\bf b}}_{(4,1,1,2,1,3)} ).
\end{split}\end{equation*}
\end{example}

{\bf Acknowledgments:} The authors thank Shun-Jen Cheng for discussions.
The first author also thanks the Institute of Mathematics of Academia Sinica in Taiwan for the hospitality and support.

\end{document}